\documentclass[12pt]{amsart}

\usepackage{amssymb,epsfig,amsfonts,setspace,dsfont,xcolor}
 \usepackage{amsmath,amscd,amsthm,array,geometry}

\newtheorem{thm}{Theorem}[section]

\newtheorem{lem}[thm]{Lemma}
\newtheorem{prop}[thm]{Proposition}

\newtheorem{cor}[thm]{Corollary}

\theoremstyle{definition}
\newtheorem{defn}[thm]{Definition}
\newtheorem{rem}[thm]{Remark}
\newtheorem{exam}[thm]{Example}

\newcommand{\G}{\mathbf{G}}
\newcommand{\Galois}{\Gamma_{E/F}}

\newcommand{\vol}{\mathrm{vol}}
\newcommand{\im}{\mathrm{im}}
\newcommand{\loc}{\mathcal{L}}
\newcommand{\M}{\mathcal{M}}
\newcommand{\tr}{\mathrm{tr}}
\newcommand{\Hom}{\mathrm{Hom}}
\newcommand{\ind}{\mathrm{ind}}
\newcommand{\g}{\mathfrak{g}}
\newcommand{\comp}{\mathfrak{k}}

\newcommand{\adele}{\mathbb{A}}
\newcommand{\finadele}{\mathbb{A}^{\mathrm{fin}}}

\newcommand{\R}{\mathbb{R}}
\newcommand{\C}{\mathbb{C}}
\newcommand{\PGL}{\mathbf{PGL}}
\newcommand{\GL}{\mathbf{GL}}
\newcommand{\MS}{\mathrm{MS}}
\newcommand{\RR}{\mathbb{R}}

\begin{document}

    \title{The Equivariant Cheeger-M\"{u}ller Theorem on Locally Symmetric Spaces}
    \author{Michael Lipnowski}
\address{Mathematics Department \\
Duke University, Box 90320 \\
Durham, NC 27708-0320, USA}
\email{malipnow@math.duke.edu}
\maketitle

\textcolor{blue}{What follows, from the abstract to Appendix \ref{nrtappendix} inclusive,  faithfully reflects the published version of the titular paper \footnote{ Lipnowski. \emph{The equivariant Cheeger-M\"{u}ller theorem on locally symmetric spaces}. Journal of the Institute of Mathematics of Jussieu  15 (2016), no. 1, 165-202.}.  Some minor corrections are noted in-text in blue.  More serious corrections, clarifications, and improvements are addressed in Appendix \ref{corrections}.  To the best of our knowledge, Corollary \ref{twistedrtlss} was unaffected by the shortcomings of the published version.  Corollary \ref{twistedatequalstwistedrt} is considerably streamlined by the results of Appendix \ref{errortermzero}; the improvements of Appendix \ref{errortermzero} render most of \S \ref{dRecmt} and \S \ref{product} unnecessary.  As Appendix \ref{errortermzero} explains, the key equality
\begin{equation*}
\text{equivariant analytic torsion} = \text{equivariant Reidemeister torsion}
\end{equation*}
holds \emph{much more generally} than under the hypotheses of Corollary \ref{twistedatequalstwistedrt}.  These improvements became possible upon understanding the error term in the Bismut-Zhang formula (stated in Theorem \ref{BZ}) more satisfactorily.  Readers interested in Corollary \ref{twistedatequalstwistedrt} can read Appendix \ref{errortermzero} in consultation with \S \ref{statementofBZ}.}

\begin{abstract}
In this paper, we provide a concrete interpretation of equivariant Reidemeister torsion and demonstrate that Bismut-Zhang's equivariant Cheeger-M\"{u}ller theorem simplifies considerably when applied to locally symmetric spaces.  In a companion paper, this allows us to extend recent results on torsion cohomology growth and torsion cohomology comparison for arithmetic locally symmetric spaces an equivariant setting. 
\end{abstract}

\setcounter{section}{-1}

\tableofcontents

\section{Introduction}
     
The contents of this introduction are as follows:

\begin{itemize}
\item
In $\S \ref{numericalexamples1},$ we exhibit two numerical examples of torsion in the cohomology of (arithmetic) locally symmetric spaces and describe how the Cheeger-M\"{u}ller theorem explains both examples.  The desire to prove equivariant analogues of the phenomena underlying these examples necessitates the analysis of equivariant Reidemeister torsion undertaken in this paper.

\item
In $\S \ref{twistedbvcv},$ we describe how an equivariant analogue of the Cheeger-M\"{u}ller theorem, together with the trace formula comparison of \cite{Lip2}, can in principle be used to prove equivariant analogues of the phenomenon discussed in $\S \ref{numericalexamples1}.$

\item
In $\S \ref{difficulties},$ we describe difficulties which must be resolved in order for the equivariant Cheeger-M\"{u}ller theorem to be used effectively for implementing the strategy outlined in $\S \ref{twistedbvcv}.$  

\item
In $\S \ref{maininformal},$ we state the main results of this paper, which resolve the difficulties highlighted in $\S \ref{difficulties}.$    

\item
In $\S \ref{outline},$ we outline the contents of this paper.    

\item
In $\S \ref{commonnotation1},$ we consolidate notation to be used regularly throughout.  
\end{itemize}

\subsection{The Cheeger-M\"{u}ller theorem applied to number theory}
\label{numericalexamples1}

Bergeron-Venkatesh \cite{BV} and Calegari-Venkatesh \cite{CV} have recently applied the Cheeger-M\"{u}ller to prove striking results about torsion in the cohomology of arithmetic locally symmetric spaces.  

We present two numerical examples of torsion in the cohomology of arithmetic hyperbolic 3-manifolds exhibiting behavior suggested by their results.  We then explain how the Cheeger-M\"{u}ller theorem sheds light on both examples.  

This section serves to motivate equivariant analogues of the results from \cite{BV} \cite{CV}; these are resolved in the case of ``cyclic base change for quaternion algebras" in \cite{Lip2} with the help of the results proven in this paper.         

\subsubsection{Growth of torsion in cohomology}
In the ring of integers $\mathbb{Z}[\sqrt{-2}],$ let $\mathfrak{p}$ be a prime ideal of residue degree 1 and norm $4969.$  Consider the congruence subgroup
$$\Gamma_0(\mathfrak{p}) =  \left\{ \gamma =  \left(    \begin{array}{cc}
a & b \\
c & d  
\end{array} \right) \in \PGL_2(\mathbb{Z}[\sqrt{-2}]) :  \mathfrak{p} \text{ divides } c \right\}.$$
Seng\"{u}n \cite{Sen} has computed
\begin{eqnarray*}
\Gamma_0(\mathfrak{p})^{\textrm{ab}} &\cong& \left( \mathbb{Z} / 2728733329370698225919458399 \right) \\ 
&\oplus& \left( \mathbb{Z} / 114525595847400940348788195788260381871 \right) \oplus ...
\end{eqnarray*}
where the above two large integers are prime.  Furthermore, the data computed in \cite{Sen} suggest that this behavior among congruence subgroups of $\PGL_2(\mathbb{Z}[\sqrt{-2}])$ is in no way anomalous.  


To highlight a connection with the next section, we note in passing that $\PGL_2(\mathbb{Z}[\sqrt{-2}])$ acts almost freely on hyperbolic three space $\mathbb{H}^3,$ and we can make the identification
$$\Gamma_0(\mathfrak{p})^{\textrm{ab}} = H_1(M, \mathbb{Z}), M = \Gamma_0(\mathfrak{p}) \backslash \mathbb{H}^3.$$
$M$ is a cusped, hyperbolic 3-orbifold.

\subsubsection{A relationship between cohomology groups on two incommensurable hyperbolic three manifolds}
\label{balls}

Let $F$ be the cubic field $\mathbb{Q}(\theta), \theta^3 - \theta + 1 = 0,$ which has one real place $\infty.$  Let $\mathfrak{p}_5$ and $\mathfrak{p}_7$ denote the unique prime ideals of $O_F$ of norm $5$ and $7.$ \smallskip

Let $D$ denote the quaternion algebra over $F$ ramified at $\mathfrak{p}_5, \infty$ and $D'$ the quaternion algebra over $F$ ramified at $\mathfrak{p}_7,\infty.$  We let 
$\Gamma(\mathfrak{n})$ denote the norm 1 units of $D^{\times}$ lying in the Eichler order of $D$ of level $\mathfrak{n}$ and $\Gamma'(\mathfrak{n})$ the norm 1 units lying in the Eichler order of $D'$ of level $\mathfrak{n}.$  We consider the compact hyperbolic 3-orbifolds
$$W(\mathfrak{n}) = \Gamma(\mathfrak{n}) \backslash \mathbb{H}^3, W'(\mathfrak{n}) = \Gamma'(\mathfrak{n}) \backslash \mathbb{H}^3.$$ 
Letting $\mathfrak{n}_{5049}$ denote the unique ideal of $O_F$ of norm $5049,$ one can compute (see \cite[$\S 1.2$]{CV}) that if $S$ is divisible by all prime numbers less than $40,$ then
$$H_1(W(\mathfrak{p}_7 \mathfrak{n}_{5049}), \mathbb{Z} [S^{-1}]) = (\mathbb{Z} / 43)^4 \oplus (\mathbb{Z} / 61)^2 \oplus (\mathbb{Z}/ 127) \oplus (\mathbb{Z} / 139)^2 \oplus (\mathbb{Z} / 181) \oplus (\mathbb{Z} [S^{-1}])^{81} \oplus (\mathbb{Z} / 67)^2$$
$$H_1(W'(\mathfrak{p}_5 \mathfrak{n}_{5049}), \mathbb{Z}[S^{-1}]) =  (\mathbb{Z} / 43)^4 \oplus (\mathbb{Z} / 61)^2 \oplus (\mathbb{Z}/ 127) \oplus (\mathbb{Z} / 139)^2 \oplus (\mathbb{Z} / 181) \oplus (\mathbb{Z}[S^{-1}])^{113}.$$
This is striking because the groups $\Gamma(\mathfrak{p}_7 \mathfrak{n}_{5049})$ and $\Gamma'(\mathfrak{p}_5 \mathfrak{n}_{5049})$ are incommensurable; there is no natural map between $W(\mathfrak{p}_7 \mathfrak{n}_{5049})$ and $W'(\mathfrak{p}_5 \mathfrak{n}_{5049}).$

\subsubsection{The Cheeger-M\"{u}ller theorem: a common thread}
\label{cmmotivation}

Ray and Singer \cite{RS} made the amazing discovery that Reidemeister torsion $RT(M,F)$ (see \cite[$\S 1$]{Ch} or $\S \ref{rtdefn}$ for the definition) of orthogonally flat local systems $F \rightarrow M$ of real vector spaces can be computed by analytic means.  Cheeger \cite{Ch} and M\"{u}ller \cite{Mu1} independently later proved that for a compact Riemannian manifold $M,$  
$$RT(M,F) = \tau(M,F)$$
%
where $\tau(M,F)$ is the analytic torsion of $F$ (see \cite[$\S 3$]{Ch} for the definition).  Analytic torsion is an invariant of the spectrum of the Laplace operators attached to $F \rightarrow M.$ More generally, M\"{u}ller later proved in  \cite{Mu2} that, for \emph{unimodular} local systems of real (resp. complex) vector spaces $L \rightarrow M$ equipped with a metric on $L,$
\begin{equation} \label{untwistedcm}
RT(M,L) = \tau(M,L). 
\end{equation}

It is an insight of Calegari-Venkatesh \cite{CV} (see (b) below) and Bergeron-Venkatesh \cite{BV} (see (a) below) that the Cheeger-M\"{u}ller theorem can be used to great effect in number theory in at least two different ways.  For these applications, it is crucial to observe that if $L \rightarrow M$ is a unimodular local system of free abelian groups equipped with a metric on $L_\R \rightarrow M,$ then
\begin{equation} \label{concretert}
RT(M,L) = \frac{|H^1(M,L)_{\mathrm{tors}}|| H^3(M,L)_{\mathrm{tors}} | \cdots}{|H^0(M,L)_{\mathrm{tors}}| |H^2(M,L)_{\mathrm{tors}}| \cdots} \times \frac{R^0(M,L) R^2(M,L) \cdots}{R^1(M,L)R^3(M,L) \cdots}
\end{equation}

where $R^i(M,L)$ equals the volume $H^i(M,L)_{\mathrm{free}}$ (see $\S \ref{examples},$ example (3)).
\medskip

\begin{itemize}
\item[(a)]
Let $\G$ be an algebraic group over $\mathbb{Q}.$  The real group $\G(\mathbb{R})$ is (essentially) the isometry group of a symmetric space $X$ of non-compact type.  Fix a rational representation $\rho: \G \rightarrow \GL(V);$ let $\mathcal{O} \subset V$ be a $\mathbb{Z}$-lattice and $\G(\mathbb{Z})$ its stabilizer.  As explained in \cite[\S 2]{Lip2},
for appropriate sequences of spaces $Y_n = \Gamma_n \backslash X$ with $\Gamma_n \subset \G(\mathbb{Q})$ an arithmetic subgroup, e.g. a subgroup of $\G(\mathbb{Z})$ cut out by finitely many congruence conditions, the representation $\rho$ gives rise to a compatible sequence $L_n \rightarrow Y_n$ of local systems of free $\mathbb{Z}$-modules.  Bergeron and Venkatesh are able to show that for appropriate representations $\rho,$ provided the injectivity radius of $Y_n$ approaches infinity, the spectral invariants $\log \tau(Y_n,L_n) / \vol(Y_n)$ converge to a spectral invariant $\tau^{(2)}(X,\rho_{\R})$ of the pair $(X, \rho_\R).$  By finding explicit examples $(X, \rho)$ where $\tau^{(2)}(X,\rho_{\R}) \neq 0,$ they prove that  
\begin{equation} \label{limitmultiplicity}
\log RT(Y_n,L_n) / \vol(Y_n) =  \log \tau(Y_n,L_n) / \vol(Y_n) \rightarrow \tau^{(2)}(X,\rho_\R) \neq 0
\end{equation}
as a consequence of the Cheger-M\"{u}ller theorem.  Restricting to representations $\rho$ for which $\log R^\bullet(Y_n, L_n) = 0$ provides examples for which torsion in the cohomology groups of $L_n \rightarrow Y_n$ grows exponentially with the volume.\smallskip  

Hyperbolic three space $\mathbb{H}^3$ happens to satisfy $\log \tau(\mathbb{H}^3, \rho_\R) \neq 0$ for every $\rho.$  Though the proof of Bergeron-Venkatesh does not apply to the trivial local system (see \cite[$\S 4$]{BV}, especially the discussion of ``strongly acyclic" representations), the fact that $H_1(\Gamma_0(\mathfrak{p}) \backslash \mathbb{H}^3, \mathbb{Z})_{\mathrm{tors}}$ is so large when the norm of  $\mathfrak{p}$ equals $4969$ is a shadow of \eqref{limitmultiplicity}, which is still expected to be true.

\item[(b)]
Though the hyperbolic 3-orbifolds $M = W(\mathfrak{p}_7 \mathfrak{n}_{5049})$ and $N = W'(\mathfrak{p}_5 \mathfrak{n}_{5049})$ are incommensurable, they have very closely related length spectra.  Using the trace formula, one can show that correspondingly the spectra of the Laplace operators of $M$ and $N$ are very closely related; the definitive generalization of this spectral comparison was proven by Jacquet and Langlands in \cite{JL}.

Calegari and Venkatesh use this fact to relate $\tau(M, \underline{\mathbb{Z}})$ and $\tau(N, \underline{\mathbb{Z}}).$  In conjunction with the Cheeger-M\"{u}ller theorem, this implies a very close relationship between the sizes of $H_1(M, \mathbb{Z})_{\mathrm{tors}}$ and $H_1(N, \mathbb{Z})_{\mathrm{tors}}.$  Proving generalizations of the second computational example in this manner forms the content of their book \cite{CV}.
\end{itemize}

\subsection{Motivation: cohomology growth and cohomology relationships in the equivariant setting}
\label{twistedbvcv}

Let $D$ be any quaternion algebra over a number field $F$ with $a$ real places and $b$ complex places.  Let $E/F$ be a cyclic Galois extension of prime degree $p$ with Galois group $\Galois = \langle \sigma \rangle.$  Let $\Gamma$ be an appropriately chosen congruence subgroup of $D^{\times}/F^{\times}, \Gamma'$ a carefully chosen Galois-stable ``matching" subgroup of $D_E^{\times} / E^{\times},$ and $L$ and $\loc$ be ``matching local systems" (see \cite[\S 4.3]{Lip2}). 
The way in which $\Gamma$ and $\Gamma'$ must be related is the subject of \cite[\S 5]{Lip2}, \cite[\S 6]{Lip2}.
The local systems $\loc, L$ are described in \cite[\S 7.1]{Lip2}. 
Notably, they are local systems of free $O_N$-modules, for an appropriate finite extension $N/F,$ for which $L_N, \loc_N$ are acyclic.  Let $M = \Gamma \backslash (\mathbb{H}^2)^a \times (\mathbb{H}^3)^b$ and let $\M = \Gamma' \backslash (\mathbb{H}^2)^{ap} \times (\mathbb{H}^3)^{bp}.$  The local system $\loc \rightarrow \M$ is equivariant for the natural $\Galois$ action. \smallskip

In the companion paper \cite{Lip2}, we use trace formula methods to prove an identity of the shape 
\begin{equation} \label{spectralcomparison}
\tau_{\sigma}(\M, \loc_\iota) = \tau(M,L_\iota)^p.
\end{equation}
for each embedding $\iota: N \hookrightarrow \C.$  In this equation, $\tau_{\sigma}(\M, \loc_{\iota})$ denotes equivariant analytic torsion (see Definition \ref{defntwistedrt}); $\tau_{\sigma}(\M, \loc_{\iota})$ is a spectral invariant of $\loc_{\iota}$ together with the action of $\langle \sigma \rangle.$  One hopes that, by an equivariant analogue of the Cheeger-M\"{u}ller theorem, \eqref{spectralcomparison} will yield
\begin{equation} \label{rtcomparison} 
RT_{\sigma}(\M, \loc_{\iota}) = RT(M,L_{\iota})^p,
\end{equation}
where $RT_{\sigma}(M,L_{\iota})$ denotes equivariant Reidemeister torsion (see Definition \ref{morsesmalert}).  Let us assume the validity of \eqref{rtcomparison}.  Combining \eqref{spectralcomparison} and \eqref{rtcomparison} with known results on the growth of $\tau(M,L_{\iota})$ (see \cite{BV}) would yield growth results for $RT_{\sigma}(\M, \loc_\iota).$  An equivariant analogue of \eqref{concretert} could be directly combined with \eqref{rtcomparison} to yield a numerical comparison of torsion cohomology.

\subsection{Difficulties with the equivariant Cheeger-M\"{u}ller theorem}
\label{difficulties}
Lott-Rothenberg in  \cite{LR} and L\"{u}ck in \cite{Lu} have proven (see Theorem \ref{luck}) that
$$\tau_{\sigma}(\M, \loc_{\iota}) = RT_{\sigma}(\M,\loc_{\iota})$$
for \emph{unitarily flat} equivariant local systems $\loc_{\iota} \rightarrow \M.$  However, other than the trivial local system, the most natural local systems $\loc$ on spaces such as $\Gamma' \backslash (\mathbb{H}^2)^{ap} \times (\mathbb{H}^3)^{bp}$ are not unitarily flat.  Those which are unitarily flat are not acyclic.  

In the interest of generalizing the torsion growth theorem (a) and the numerical torsion comparison (b) discussed in $\S \ref{cmmotivation},$ we are very fortunate to have the Bismut-Zhang formula available; this is discussed at length in $\S \ref{dRecmt}.$  Bismut and Zhang in \cite{BZ2} prove that 
$$\log \tau_{\sigma}(\M, \loc_\iota) - \log RT_{\sigma}(\M, \loc_\iota) = E(\M, \loc),$$
where $E(\M,\loc)$ is an error term which localizes to the fixed point set $\M_{\sigma}.$  The expression for $E$ from \cite{BZ2} is unfortunately not explicit enough for the aforementioned growth and numerical comparison applications.  

In addition to the question of whether $E(\M,\loc) = 0,$ there is a futher issue of interpreting ``what $RT_{\sigma}(\M,\loc)$ is."  For the applications sketched in $\S \ref{twistedbvcv},$ we require an analogue of \eqref{concretert} giving a concrete interpretation for Reidemeister torsion in terms of sizes of torsion cohomology groups.

The main purpose of this paper is to resolve these two issues.

\subsection{Statement of main results}
\label{maininformal}
We quickly set some notation to state our first main result.
Let $p$ be prime.  Let $P(x) = x^{p-1} + x^{p-2} + ... + 1,$ the $p$-cyclotomic polynomial.  

Let $R$ be a commutative ring.  For any $R$-module $A$ acted on $R$-linearly by $\langle \sigma \rangle \cong \mathbb{Z} / p\mathbb{Z}$ and any polynomial $h \in R[x],$ let 
$$A^{h(\sigma)} := \{ a \in A: h(\sigma) \cdot a = 0 \}.$$ 
\newtheorem*{concretetwistedrt}{Corollary \ref{twistedrtlss} (concrete interpretation of twisted Reidemeister torsion)}
\begin{concretetwistedrt}
Let $\loc \rightarrow M$ be a rationally acyclic, metrized, unimodular local system $\loc \rightarrow \M$ of free abelian groups acted on isometrically by $\langle \sigma \rangle \cong \mathbb{Z} / p\mathbb{Z}.$  Suppose that the fixed point set $\M_{\sigma}$ has Euler characteristic 0.  Then
\begin{eqnarray*}
\log RT_{\sigma}(\M,\loc) &=&  \textcolor{blue}{- \sum (-1)^i \left( \log \left|H^{i}(\M, \loc)[p^{-1}]^{\sigma - 1} \right|  -  \frac{1}{p-1} \log \left|H^{i}(\M, \loc)[p^{-1}]^{P(\sigma)} \right| \right)} \\
&+& O\left( \log|H^{*}(\mathcal{M}, \loc)[p^{\infty}]| +\log |H^{*}(\mathcal{M}, \loc_{\mathbb{F}_p})| + \log |H^{*}(M, \loc_{\mathbb{F}_p})| \right)
\end{eqnarray*}
\end{concretetwistedrt}

This is proven by relating the ``naive twisted Reidemeister torsion" $NRT_{\sigma}(C^{\bullet}(\M,\loc;K))$ (see Definition \ref{nrt}) of the cochain complex $C^{\bullet}(\M, \loc;K)$ to the twisted Reidemeister torsion of $(\M, \loc)$ (see Proposition \ref{rtvsnrt}) and then using a spectral sequence argument to relate $NRT_{\sigma}$ to the cohomology of $H^\bullet(\M, \loc)$ (see Proposition \ref{finitecomplexestimate}) in a ``triangulation independent" manner.

The next theorem describes circumstances under which twisted Reidemeister torsion equals twisted analytic torsion.  

\newtheorem*{atequalsrt}{Corollary \ref{twistedatequalstwistedrt} ($\tau_{\sigma}$ often equals $RT_{\sigma}$)}

\begin{atequalsrt}
Let $\loc \rightarrow \M$ be an equivariant, metrized, rationally acyclic local system of free abelian groups over a locally symmetric space $\M$ acted on equivariantly and isometrically by $\langle \sigma \rangle \cong \mathbb{Z} / p\mathbb{Z}.$  Suppose further that the restriction to the fixed point set $\loc|_{\M_{\sigma}} = L^{\otimes p}$ for $L \rightarrow \M_{\sigma}$ self-dual and that $\M_{\sigma}$ is odd dimensional.  Then 
$$\log \tau_{\sigma}(\M, \loc) = \log RT_{\sigma}(\M,\loc).$$
\end{atequalsrt}

This theorem is proven in two steps:
\begin{itemize}
\item
Proposition \ref{finalmariage} compares the error terms of two different applications of the Bismut-Zhang formula \eqref{BZformula}, one for $\loc \rightarrow \M$ and one for $L^{\boxtimes p} \rightarrow \M_{\sigma}^p$ to prove that  
$$\log RT_{\sigma}(\M, \loc) - \log \tau_{\sigma}(\M, \loc) = \log RT_{\sigma}(\M_{\sigma}^p, L^{\boxtimes p}) - \log \tau_{\sigma}(\M_{\sigma}^p,L^{\boxtimes p}).$$
\item
Section \ref{product} proves that the latter difference
$$\log RT_{\sigma}(\M_{\sigma}^p, L^{\boxtimes p}) - \log \tau_{\sigma}(\M_{\sigma}^p,L^{\boxtimes p})$$
is zero by separately proving that 
$$\log \tau_{\sigma}(\M_{\sigma}^p, L^{\boxtimes p}) = p \log \tau(\M_{\sigma},L)$$
in Proposition \ref{torsionproduct},
$$\log RT_{\sigma}(\M_{\sigma}^p, L^{\boxtimes p}) = p \log RT(\M_{\sigma},L)$$
in Theorem \ref{productcheegermuller}, and then applying the untwisted Cheeger-M\"{u}ller theorem to conclude.
\end{itemize}

\subsection{Acknowledgements}

This paper is an outgrowth of the author's PhD thesis.  It owes its existence to the inspirational work of Bergeron-Venkatesh \cite{BV} and Calegari-Venkatesh \cite{CV}.  

The author thanks Jayce Getz, Les Saper, and Mark Stern for their helpful comments on drafts of this paper. 
 
The author would like to thank Nicolas Bergeron for many stimulating discussions on torsion growth and twisted endoscopy.  He would also like to thank Jean-Michel Bismut for patiently explaining his work with Zhang \cite{BZ2} on the equivariant Cheeger-M\"{u}ller theorem. 

Last but not least, the author would like to express his deep gratitude to his advisor, Akshay Venkatesh, for sharing so many of his ideas and for providing constant encouragement and support during the preparation of this work.

\subsection{Outline}
\label{outline}

\begin{itemize}

\item
In $\S \ref{chaincomplex},$ we recall the definition of twisted analytic torsion and twisted Reidemeister torsion of an (equivariant) metrized local systesm $\loc \rightarrow \M$ of $O_N$-modules over a compact Riemannian manifold, where $O_N$ is the ring of integers of a number field $N.$  In $\S \ref{geomtwistedrtorsion},$ we prove a ``compatibility with restriction of scalars" property of twisted Reidemeister torsion for equivariant local systems of $O_N$-modules, $N$ a number field.  This property is used to relate our main results, which concern comparisons of Reidemeister torsion for matching local systems $\loc \rightarrow \M$ and $L \rightarrow M$ of $O_N$-modules over different manifolds, to a relationship between sizes of cohomology groups 
(see \cite[\S 7.2]{Lip2}). \smallskip

In $\S \ref{luckcm},$ we state L\"{u}ck's version of the equivariant Cheeger-M\"{u}ller theorem from \cite{Lu}.  This theorem, along with the untwisted Cheeger-M\"{u}ller theorem for unimodular local systems proven in \cite{Mu2}, can be used in conjunction with the spectral comparison theorem proven in \cite[\S 4.4]{Lip2} 
to yield numerical cohomology comparisons.

\item
In $\S \ref{finitechaincomplex},$ we discuss a version of the twisted Cheeger-M\"{u}ller theorem valid for a finite, rationally acyclic metrized chain complex $A^{\bullet}$ of free abelian groups acted on isometrically by $\sigma$ with $\sigma^p = 1.$  This calculation naturally leads to the definition of ``naive equivariant Reidemeister torsion" (see Definition \ref{nrt}).   

\item
In $\S \ref{preliminaryestimate},$ in the case where $A^{\bullet} = C^{\bullet}(\M, \loc; K)$ is the group of $\loc$-valued cochains on $\M$ with respect to a fixed equivariant triangulation,  we carry through a spectral sequence argument to relate the naive Reidemeister torsion, up to a controlled error, to a quantity which is patently related to the structure of $H^{*}(\M, \loc)$ as a $\sigma$-module and which is independent of the triangulation $K.$  Finally, in $\S \ref{smalldifference},$ for $\loc \rightarrow \M$ an equivariant, unimodular local system over a locally symmetric space, we relate the naive twisted Reidemeister torsion of $C^{\bullet}(\M, \loc; K)$ to its twisted Reidemeister torsion. 

\item
In $\S \ref{dRecmt},$ we recall the statement of the equivariant Cheeger-M\"{u}ller theorem, proven by Bismut and Zhang.  Let $\loc \rightarrow \M$ be a $\sigma$-equivariant metrized local system.  The Bismut-Zhang formula enables us to prove Proposition $\ref{finalmariage},$ which shows that the difference between the twisted Reidemeister torsion and the twisted analytic torsion of a $\sigma$-equivariant local system $\loc \rightarrow \M$ equals the same difference for a $L^{\boxtimes p} \rightarrow (\M_{\sigma})^p,$ where $\sigma$ acts on the latter by cyclic shift.  This constitutes progress since the twisted analytic torsion and twisted Reidemeister torsion of a product are individually computed in $\S \ref{product}.$   

\item
In $\S \ref{product},$ we study the equivariant analytic torsion and equivariant Reidemeister torsion for a product.  We directly calculate both the equivariant analytic torsion of a metrized local system and the ``naive Reidemeister torsion" (see Definition \ref{nrt}) of the local system $L_{\rho}^{\boxtimes n} \rightarrow \mathcal{M} = M^n$ over the compact Riemannian manifold $M^n$ with respect to the cyclic shift $\sigma.$  In $\S \ref{smalldifferenceproduct},$ we show that for many unimodular metrized local systems $L \rightarrow M$ which are not necessarily unitarily flat, the conclusion of L\"{u}ck's variant of the twisted Cheeger-M\"{u}ller theorem continues to hold for $L^{\boxtimes p} \rightarrow M^p$: its twisted Reidemeister torsion often equals its twisted analytic torsion.  
\end{itemize}

\subsection{Notation used throughout}
\label{commonnotation1}

This section compiles a list of frequently used notation.  The descriptions given are consistent with the most common usage of the corresponding symbols.  The reader should be warned, however, that within a given chapter or section, the below symbols might carry a slightly different meaning; such local changes of notation will be made clear as necessary.   
\begin{itemize}
\item
$L \rightarrow M$ denotes a local system of projective $O_F,F, \mathbb{Q}, \mathbb{Z}, \R,$ or $\C$-modules, depending on the context.

\item
$\loc \rightarrow \M$ denotes a local system equivariant for the action of a finite group $\Gamma,$ usually $\Gamma = \langle \sigma \rangle$ with $\sigma^p = 1.$

\item
$C^{\bullet}(M,L; K)$ denotes the complex of $L$-valued cochains on $M$ with respect to a triangulation $K$

\item
Let $f$ be a Morse function on a smooth closed manifold $M$ and $X$ a weakly gradient-like vector field with respect to $f.$  Let $\Phi_X$ denote the flow generated by $X.$  For every critical point $p$ of $X,$ we let $W^u(p)$ denote the unstable (ascending) manifold of $p$ and $W^s(p)$ denote the stable (descending) manifold of $p.$  These are respectively defiend as
\begin{align*}
W^u(p) &= \{m \in M: \lim_{t \rightarrow -\infty} \Phi_{X,t}(m) = p \} \\
W^s(p) &= \{m \in M: \lim_{t \rightarrow \infty} \Phi_{X,t}(m) = p \}.
\end{align*}
See \cite[\S 2.4]{Nic} for further discussion.   

\item
$\MS(X,L)$ denotes the Morse-Smale complex associated with a vector field $X$ on $M$ which is weakly gradient-like with respect to a fixed Morse function $f$ and which satisfies Morse-Smale transversality.  

\item
$RT(M, L)$ denotes the Reidemeister torsion of the cochain complex $C^{\bullet}(M,L;K)$ for a local system $L \rightarrow M$ provided the implicit volume forms and triangulation are understood.  $RT_{\sigma}(\M, \loc)$ denotes the twisted Reidemeister torsion (evaluated at $\sigma$) of the complex $C^{\bullet}(M,L;K)$ for a $\langle \sigma \rangle$-equivariant local system $\loc \rightarrow \M$ of free abelian groups.

\item
$RT(X, L)$ denotes the Reidemeister torsion of the Morse-Smale complex $\MS(X,L)$ for a vector field $X,$ satisfying Morse-Smale transversality, and a local system $L \rightarrow M,$ provided the Morse function $f$ and the implicit volume forms are understood.  $RT_{\sigma}(X,\loc)$ denotes the twisted Reidemesiter torsion of the Morse-Smale complex whenever $\loc \rightarrow \M$ is a $\langle \sigma \rangle$ equivariant local system.

\item
For an $R$-module $A$ acted on $R$-linearly by $\langle \sigma \rangle,$ we let $A^{\sigma - 1} := \{ a \in A: (\sigma - 1) \cdot a = 0 \}$ and $A^{P(\sigma)} = \{ a \in A: P(\sigma) \cdot a = 0 \}$ where $P(\sigma)$ denotes the $p$-cyclotomic polynomial $P(x) = x^{p-1} + x^{p-2} + ... + 1.$  Sometimes, we denote these by $A[\sigma - 1]$ and $A[P(\sigma)]$ as well.

\item
For an $R$-module $A$ acted on $R$-linearly by $\langle \sigma \rangle,$ we define $A' := A / (A[\sigma - 1] \oplus A[P(\sigma)]).$  Similarly, if $A^{\bullet}$ is a complex of $R$-modules acted on $R$-linearly by $\langle \sigma \rangle,$ we define $A'^{\bullet} :=  A^{\bullet} / (A^{\bullet}[\sigma - 1] \oplus A^{\bullet}[P(\sigma)]).$

\item
For a finite abelian group $B$ and a rational prime $p,$ we let $B[p^{\infty}]$ and $B[p^{-1}]$ respectively denote the $p$-primary subgroup of $B$ and the prime to $p$ subgroup  of $B.$  These are canonically isomorphic to $B \otimes_{\mathbb{Z}} \mathbb{Z}_{(p)}$ and $B \otimes_\mathbb{Z} \mathbb{Z}[p^{-1}]$ respectively. 

\item
$\sum {}^{*}, \prod {}^{*}, \otimes {}^{*}$ respectively denote alternating sum, product, and tensor product.  The alternating tensor product of $R$-modules $M_1,...,M_n$ means $M_1 \otimes_R M_2^{*} \otimes_R M_3 \otimes ...$

\item 
For a projective $R$-module $M,$ we let $\det(M)$ denote $\wedge_R^{\mathrm{top}} M.$  For a complex $M^{\bullet} := 0 \rightarrow M_0 \rightarrow M_1 \rightarrow \cdots$ of projective $R$-modules, its determinant is defined to be $\det(M) := \otimes {}^{*} \det(M_i).$ 
\end{itemize}

\section{The equivariant Cheeger-M\"{u}ller theorem}
\label{chaincomplex}



\begin{itemize}
\item
In $\S \ref{twistedAtorsion},$ we define equivariant analytic torsion of a metrized local system $\loc \rightarrow \M$ of $\C$-vector spaces acted on equivariantly by isometries by a finite group $\Gamma.$

\item
In $\S \ref{rtdefn},$ we define the Reidemeister torsion of a complex $A^{\bullet}$ of $K$-vector spaces equipped with volume forms on the chain groups $A^{\bullet}$ and the cohomology groups $H^{*}(A^{\bullet}),$ where $K$ is any field.  

\item
In $\S \ref{twistedRtorsion},$ we define the twisted Reidemeister torsion of an abstract metrized complex of $\C$-vector spaces.

\item
In $\S \ref{morsesmale},$ we define the Morse smale complex associated to a local system $\loc \rightarrow X$ together with a Morse function $f: X \rightarrow \R.$

\item
In $\S \ref{geomtwistedrtorsionms}$ and $\S \ref{geomtwistedrtorsion},$ we explain how given some auxiliary volume forms associated with a $\langle \sigma \rangle$-equivariant unimodular local system of $N$-vector spaces $\loc \rightarrow \M,$ for any field $N,$ allows us to define an $N^{\times}$-valued version of equivariant Reidemeister torsion.  In particular, for $N$ a number field, we prove in $\S \ref{restrictionofscalars}$ a ``norm compatibility" between the equivariant Reidemeister torsion of the geometric complex $C^{\bullet}(\M, \loc ;K),$ where $\loc$ is a local system of $N$ vector spaces, and the same geometric complex $C^{\bullet}(\M, \loc_{\mathbb{Q}}; K)$ where $\loc_{\mathbb{Q}}$ denotes $\loc$ viewed as a local system of $\mathbb{Q}$-vector spaces.  

\item 
In $\S \ref{luckcm},$ we finally state the twisted Cheeger-M\"{u}ller theorem due to L\"{u}ck in \cite{Lu} and the untwisted Cheeger-M\"{u}ller theorem due to M\"{u}ller in \cite{Mu1}.
\end{itemize}

\subsection{Definition of twisted analytic torsion}
\label{twistedAtorsion}

Let $\mathcal{M}$ be a compact Riemannian manifold. 

\begin{defn}
A \emph{metrized local system}  $\loc \rightarrow \M,$ is a local system of free abelian groups equipped with a metric on $\loc_\R.$  The metric on $\loc_\R$ is \emph{not} required to be compatible with the flat structure on $\loc,$ which is to say that parallel transport need not be unitary.   
\end{defn}

\begin{defn}
If the parallel transport on $\loc_\R$ induced by the flat structure is unitary, we call $\loc \rightarrow \M$ \emph{unitarily flat}.
\end{defn}

Let $\Gamma$ be a group of finite order acting equivariantly on a metrized local system $\mathcal{L} \rightarrow \mathcal{M}$ by isometries.  Let $\Delta_{j,\mathcal{L}}$ denote the $j$-form Laplacian acting on $\Omega^j(\mathcal{M}, \mathcal{L}).$ \bigskip

Note that the $\lambda$-eigenspace of $\Delta_{j,\mathcal{L}},$ call it $E_{j,\mathcal{L},\lambda},$ is preserved under pullback by $\Gamma$ because $\Gamma$ acts by isometries; each $E_{j,\mathcal{L},\lambda}$ is a representation of $\Gamma.$  Let $\mathrm{Rep}(\Gamma)$ denote the representation ring of the finite group $\Gamma.$  We can form the $\mathrm{Rep}(\Gamma)_{\mathbb{C}}$-valued equivariant zeta functions
$$\zeta_{j,\mathcal{L},\Gamma}(s) = \sum \lambda^{-s} [E_{j, \mathcal{L},\lambda}] \in \mathrm{Rep}(\Gamma)_{\mathbb{C}}$$
where the sum ranges over the non-zero eigenvalues of $\Delta_{j,\mathcal{L}}.$  All of the functions $\zeta_{j,\mathcal{L},\sigma}$ admit meromorphic continuation to the entire complex plane and are holomorphic in a neighborhood of $s = 0$ (see \cite[Lemma 1.13]{Lu}).  Form the linear combination

$$Z_{\mathcal{L},\Gamma}(s) = \frac{1}{2} \sum_j (-1)^j j \cdot \zeta_{j,\mathcal{L},\Gamma}(s).$$

\begin{defn} \label{defntwistedrt}
The \emph{twisted (or equivariant) analytic torsion $\tau_{\Gamma}$}  is defined by
$$\tau_{\Gamma}(\mathcal{L}) = Z_{\mathcal{L}, \Gamma}'(0) \in \textrm{Rep}(\Gamma)_{\mathbb{C}}.$$
If $\Gamma = \langle \sigma \rangle$ is a cyclic group, we will often denote
$$\tau_{\sigma}(\mathcal{L}) := \tr \{\sigma | \tau_{\Gamma}(\mathcal{L}) \}.$$
\end{defn}

\textbf{Remarks}. 
\begin{itemize}
\item
If $\mathcal{L} = \mathbb{Z}_\M$ is the trivial local system, then we will sometimes label our generating functions with a subscript $\M$ instead of a subscript $\mathbb{Z}_{\M}.$ 

\item
If $\sigma = 1,$ we recover the usual untwisted torsion from this definition.
\end{itemize}

\subsection{Definition of Reidemeister torsion}
\label{rtdefn}

Let $K$ be any field.  We recall the definition of the determinant of a $K$-vector space or a complex of $K$-vector spaces.

\begin{defn} \label{det}
Let $V$ be a $K$-vector space of finite dimension $n.$  We define its determinant to be $\det(V) := \wedge^n V.$  If $A^{\bullet} = 0 \rightarrow A^0 \rightarrow ... \rightarrow A^n \rightarrow 0$ is a finite complex of finite dimensional $K$-vector spaces, we define 
$$\det(A^\bullet) := \bigotimes {}^{*} \det(A^{\bullet}) := \bigotimes_{i=0}^n \det(A^i)^{(-1)^i},$$
where the superscript $-1$ denotes $K$-dual.  In words, this is the ``alternating tensor product" of the determinants of the constituents of the complex $A^{\bullet}.$
\end{defn}

Let $(A^{\bullet}, d^{\bullet})$ be a bounded complex of finite dimensional $K$-vector spaces.  Let $B^{\bullet}, Z^{\bullet}, H^{\bullet}$ be the associated complexes of coboundaries, cocycles, and cohomology.  Suppose that $A^i, B^i, Z^i, \text{ and } H^i$ have respective dimensions $a_i ,b_i, z_i, h_i.$  Suppose further that each group $A^{\bullet}, H^{*}(A^{\bullet})$ is equipped with a volume form, i.e. there are given volume forms $\mu_i \in \left( \wedge^{h_i} H^i(A^{\bullet}) \right)^{*}, \omega_i \in \left( \wedge^{a_i} A^i \right)^{*}.$  There is a canonical isomorphism
\begin{equation} \label{detofcomplex}
\det(A^{\bullet}) \otimes_K \det( H^i(A^{\bullet}) )^{-1} \cong K
\end{equation}
(see \cite[$\S 1,$ Proposition 1]{KM}).  Let the section $s_{A^{\bullet}}$ of $\det(A^{\bullet})^{-1} \otimes_K \det( H^{*}(A^{\bullet}) )$ be the preimage of 1 under the isomorphism of (\ref{detofcomplex}).  

\begin{defn}[Reidemeister torsion through volume forms] \label{rtvolumeforms}
The \emph{Reidemeister torsion} $RT(A^{\bullet}, \omega, \mu) \in K^{\times}$ is the value of $\bigotimes {}^{*} \omega_i \otimes \bigotimes {}^{*} \mu_i^{-1}$ evaluated on the section $s_{A^{\bullet}}$ of $\det(A^{\bullet})^{-1} \otimes_K \det(H^{*}(A^{\bullet}) ).$  \smallskip

The Reidemeister torsion of $A^{\bullet}$ can be computed as follows.  Choose $\rho_i \in \left( \wedge^{b_i} A^i \right)^{*}$ with $\rho_i |_{\wedge^{b_i} B^i} \neq 0$ and $\sigma_i \in \left( \wedge^{h_i} A^i \right)^{*}$ with $\sigma |_{\wedge^{h_i} Z^i} = \pi^{*}(\mu_i), \pi$ denoting the projection $Z^\bullet \rightarrow H^\bullet.$  Then

$$\rho_i \wedge d_{i+1}^{*}(\rho_{i+1}) \wedge \sigma_i = m_i \omega_i$$

for some $m_i \in K^{\times}.$  Then

$$RT(A^{\bullet}, \omega, \mu) = \prod {}^{*} m_i \in K^{\times}$$
\end{defn}

(cf. \cite[$\S 1$]{Ch}).

\subsubsection{Norms versus volume forms for complexes of $\mathbb{C}$-vector spaces}
\label{normsvsvolumeforms}

Let $A^{\bullet}$ be a finite complex of finite $\mathbb{C}$-vector spaces.  We introduce some terminology.

\begin{defn}[norms and volume forms]
Let $V$ be a $\C$-vector space of finite dimension $n.$  By a \emph{norm on $V$}, we mean a non-degenerate norm on the one dimensional complex vector space $\wedge^n V.$  A \emph{volume form} denotes a non-zero element of $(\wedge^n V)^*.$  The absolute value of a volume form on $V$ is a norm on $V.$    
\end{defn}

\begin{defn}[Reidemeister torsion through norms]
Suppose each $A^i$ is endowed with a norm $\alpha_i$ and each $H^i(A^{\bullet})$ is endowed with a norm $\beta_i.$  We can choose $\alpha_i' \in \left(\wedge^{\mathrm{top}} A^i \right)^{*}$ and $\beta_i' \in \left(\wedge^{\mathrm{top}} H^i(A^{\bullet}) \right)^{*}$ for which $|\alpha_i'| = \alpha_i, |\beta_i'| = \beta_i.$  We define

$$RT(A^{\bullet}, \alpha, \beta) = |RT(A^{\bullet}, \alpha', \beta')|$$
\end{defn}

\begin{rem}
The choice of $\alpha_i', \beta_i'$ is only ambiguous up to a complex number of absolute value 1, so the absolute value $|RT(A^{\bullet}, \alpha', \beta')|$ is independent of all choices. 
\end{rem}

\subsubsection{Some useful examples and properties of Reidemeister torsion}
\label{examples}

\begin{itemize}
\item[(1)]
Any finite-free abelian group $B$ gives rise to a canonical norm on $\det(B_{\mathbb{C}}),$ namely that which assigns $e_1 \wedge ... \wedge e_r$ norm 1 for any basis $e_1,...,e_r$ of $B.$  For any complex $A^{\bullet}$ of finite free abelian groups, we let $\alpha_{\mathbb{Z}}$ and $\beta_{\mathbb{Z}}$ denote these norms arising from the integral structure in this manner.  One computes that
$$RT(A^{\bullet}, \alpha_{\mathbb{Z}}, \beta_{\mathbb{Z}}) = \prod {}^{*} |H^i(A^{\bullet})_{\textrm{tors}} |$$ 
(see \cite[$\S 1$]{Ch}).
\item[(2)]
If $\beta'$ is a different choice of norms on $H^i(A^{\bullet}_{\mathbb{C}})$ with $\beta_i' = k_i \beta_i,$ then
\begin{equation} \label{changeofvolumeform}
RT(A^{\bullet},\alpha, \beta') = RT(A^{\bullet}, \alpha, \beta) \times \prod {}^{*} k_i.
\end{equation}
\item[(3)]
Suppose that $L \rightarrow M$ is a metrized local system of free $\mathbb{Z}$-modules over a compact Riemannian manifold ($M,g$). Let $A^{\bullet}$ denote the group of $L$-valued cochains on $M$ with respect to a fixed choice of triangulation; the cohomology of this chain complex computes $H^{\bullet}(M, L).$  The Riemannian metric induces metrics on each $H^i(A^{\bullet}) \cong H^i(M,L)$ via Hodge theory, and hence norms $\beta_{i,g}$ on the $\det(H^i(A^{\bullet})).$  In such geometric situations, we define 
$$RT(A^{\bullet}) := RT(A^{\bullet}, \alpha_{\mathbb{Z}}, \beta_g).$$
Note that $\beta_{i,g} = \vol(H^i(M,L_{\mathbb{R}}) / H^i(M,L)) \times \beta_{i, \mathbb{Z}}.$  So by equation (\ref{changeofvolumeform}),
$$RT(A^{\bullet}) = \prod {}^{*} |H^i(M,L)_{\textrm{tors}}| \times \prod {}^{*} \vol(H^i(M,L_{\mathbb{R}}) / H^i(M,L)).$$
\item[(4)]
In our imminent discussion of Reidemeister torsion, it will be necessary to think of Reidemeister torsion in terms of (hermitian) metrics on $A^{\bullet}.$  Any such metric $h$ induces a collection of norms $\alpha_h,$ where $\alpha_h^i$ is the norm that $h$ induces on $A^i \otimes \mathbb{C}.$  \medskip

Suppose that $A^{\bullet} = C^{\bullet}(M, L;K)$ for some triangulation $K$ of the manifold $M$ from $(3).$  The group of cochains is generated by ``indicator cochains" $\mathbf{1}_{C, s},$ i.e. those which assign some cell $C$ a global section $s$ of $L|_C.$  Assume that $L \rightarrow M$ is unitarily flat.  We define a metric $h_{\mathbb{Z}}$ on $A^{\bullet}$ by insisting that 
$$h_{\mathbb{Z}}(1_{C,s},1_{C,s}) = || s ||^2 \text{ and } h_{\mathbb{Z}}(\mathbf{1}_{C,s}, \mathbf{1}_{C',s'}) = 0 \text{ if } C, C' \text{ are distinct},$$
where for a section $s \in L(C),$ we define $||s|| = ||s_x||$ for any $x \in C.$  Because $L$ is unitarily flat, this is well-defined.  Provided that $|| e_1 \wedge ... \wedge e_n || = 1$ for every integral basis $e_1,...,e_n$ of $L(C)$ and every cell $C$ of the triangulation $K,$  
$$RT(A^{\bullet}, \alpha_{h_{\mathbb{Z}}}, \beta) = RT(A^{\bullet}, \alpha_{\mathbb{Z}}, \beta).$$
\end{itemize}

\subsection{Definition of twisted Reidemeister torsion for an abstract metrized complex $A^{\bullet}$}
\label{twistedRtorsion}

Let $A^{\bullet}$ be a finite complex of finite dimensional $\C$-vector spaces together with metrics $h_i$ on $A^{\bullet}$ and induced metrics $g_i$ on each $H^i(A^{\bullet}).$ 

\begin{defn}
Let $\Gamma$ be a finite group acting on $A^{\bullet}$ by isometries.  The \emph{equivariant Reidemeister torsion of} $A^{\bullet}$ is defined by the formula
$$\log RT_{\Gamma}(A^{\bullet}, h, g) := \sum_{\pi \in \widehat{\Gamma}} \frac{1}{\dim \pi} \log RT(A^{\bullet}[\pi], h|_{A^{\bullet}[\pi]},  g|_{H^i(A^{\bullet}[\pi])} ) \cdot \pi \in \mathrm{Rep}(\Gamma)_{\mathbb{C}},$$ 
where $A^{\bullet}[\pi]$ denotes the $\pi$-isotypic subcomplex of $A^{\bullet}$ and $\mathrm{Rep}(\Gamma)$ denotes the representation ring of $\Gamma.$ 
\end{defn}

\begin{rem}
The representation ring $\mathrm{Rep}(\Gamma)$ is free as a $\mathbb{Z}$-module with basis given by the isomorphism classes of irreducible complex representations of $\Gamma.$  The map assigning to each representation its character defines an isomorphism between $\mathrm{Rep}(\Gamma)_{\C}$ and the complex-valued class functions on $\Gamma.$  We can and will use this isomorphism to evaluate elements of $\mathrm{Rep}(\Gamma)_\C$ on conjugacy classes of $\Gamma.$  
\end{rem}

Let $N$ be a number field with ring of integers $O_N.$ 

\begin{defn} \label{firstdefnrt}
Let $\loc \rightarrow \M$ be a local system of $O_N$-modules and let $\iota: N \hookrightarrow \C$ be an embedding.  Suppose $\loc_\iota = \loc \otimes_\iota \mathbb{C}$ is endowed with a unitarily flat metric.  Suppose the finite group $\Gamma$ acts equivariantly on $\loc_\iota \rightarrow \M$ by isometries.  Let $K$ be a fixed $\Gamma$-equivariant triangulation with cochain group $C^{\bullet}(\M,\loc; K).$  We define
\begin{equation} \label{geomtwistedrt}
\log RT_{\Gamma}(\M, \loc_{\iota} ; K) := \log RT_{\Gamma}(C^{\bullet}(\M,\loc_\iota;K)),
\end{equation}
where the metrics implicit on the right side of equation (\ref{geomtwistedrt}) are those induced on $C^{\bullet}(\M,\loc_\iota; K)$ and $H^{*}(C^{\bullet}(\M,\loc_\iota;K))$ by $\loc_\iota$ (see example $(4)$ from $\S \ref{examples}$).  For notational shorthand, if the group $\Gamma$ is understood - often $\Gamma = \langle \sigma \rangle$ - we define
$$RT_{\sigma}(\M,\loc_\iota;K) := RT_{\Gamma}(\M,\loc_\iota;K)(\sigma).$$ 
If $\M,$ the triangulation, and the complex embedding $\iota$ are all understood, we denote this by $RT_{\sigma}(\loc).$  
\end{defn}

\subsection{The Morse-Smale complex and the definition of twisted Reidemeister torsion}
\label{morsesmale}

\begin{defn} Let $f: M \rightarrow \R$ be a Morse function on a compact manifold $M.$  Let $X$ be a vector field on $M.$  We say that $X$ is a \emph{weakly gradient-like vector field associated to $f$} if the critical points of $X$ equal the critical points of $f$ and if $X_p(f) > 0$ for all non-critical points $p \in M.$  That is, $f$ increases along the flow of $X.$
\end{defn}

\begin{defn}
Let $X$ be a weakly gradient-like vector field associated to a Morse function $f: \M \rightarrow \R.$  We say that $X$ satisfies \emph{Morse-Smale transversality} if for every pair $p,q$ of critical points of $X,$ the ascending manifold $W^u(p)$ of $p$ and the descending manifold $W^s(q)$ of $q$ intersect transversely.
\end{defn}

Let $X$ be a weakly gradient-like vector field associated to a Morse function $f.$  Such a vector field $X$ satisfying Morse-Smale transversality gives rise to a Morse-Smale complex which computes the cohomology of any local system $\loc \rightarrow \M$ of finite projective $R$-modules, where $R$ is any ring.  The following formulation is taken from \cite[$\S 1 (c)$]{BZ} and repeated here for convenience. Define a chain complex 
$$\MS^i(X, \loc) = \bigoplus_{x \in \mathrm{Crit}(X), \mathrm{ind}(x) = i} R[W^u(x)] \otimes_{R} \loc_x,$$
where $R[W^u(x)]$ denotes the free rank-$1$ $R$-module with basis $[W^u(x)]$ and $\mathrm{ind}(x)$ denotes the Morse index of the critical point $x.$
Because $X$ satisfies Morse-Smale transversality, the set of flow lines $\Gamma(x,y)$ from $x$ to $y$ for each pair of critical points $x,y$ is finite and empty unless $\mathrm{ind}(y) = \mathrm{ind}(x) + 1.$ 

Consider a pair of critical points $x,y$ with $\mathrm{ind}(y) = \mathrm{ind}(x) + 1.$  Fix an orientation on $X$ and on each unstable manifold $W^u(x);$  this determines an orientation on each $W^s(x).$  Because $W^u(x)$ and $W^s(y)$ intersect transversally, we can ``flow their orientations" along any integral curve $\gamma$ from $x$ to $y$ to obtain a well-defined number $n_{\gamma}(x,y) = \pm 1,$ where the sign is $+1$ if the flowed orientations agree and $-1$ if they are opposite.  We define a boundary map

$$\delta: \MS^i(X, \loc) \rightarrow \MS^{i+1}(X,\loc)$$
$$\delta(W^u(x) \otimes a) = \sum_{y \in \mathrm{Crit}(X), \mathrm{ind}(y) = i+1} \sum_{\gamma \in \Gamma(x,y)} n_{\gamma}(x,y) \cdot W^u(y) \otimes PT_{\gamma}(a),$$

where $PT_{\gamma}(a)$ denotes the parallel transport of $a$ along the integral curve $\gamma.$   Some comments are in order:

\begin{itemize}
\item
The complex $\MS^{\bullet}(X, \loc)$ computes the cohomology of $\loc \rightarrow \M.$ 

\item
 Suppose that $f$ is an invariant Morse function, $X$ is an invariant weakly gradient-like vector field, and $R = \R.$  Let $\loc \rightarrow \M$ be a $\Gamma$-equivariant, metrized local system.  Assign the vector space $\bigoplus_{\mathrm{ind}(x) = i} \R [W^u(x)]$ the combinatorial metric where the $[W^u(x)]$ form an orthonormal basis.  Then $\Gamma$ acts on $\MS^{\bullet}(X, \loc)$ by isometries and so the equivariant Reidemeister torsion of this complex makes sense.
\end{itemize}

\begin{defn} \label{morsesmalert}
Let $\loc$ be a metrized local system.  We define $RT_{\Gamma}(X, \loc)$ to be the $\Gamma$-equivariant Reidemeister torsion of the metrized $\Gamma$-complex $\MS^{\bullet}(X,\loc).$
\end{defn}

\subsection{Reidemeister torsion of the $\sigma$-isotypic pieces of $\MS^{\bullet}(X,\loc)$}
\label{geomtwistedrtorsionms}

In $\S \ref{morsert}$ and $\S \ref{cochainrt},$ we specialize Definition \ref{rtvolumeforms} for Reidemeister torsion to isotypic pieces of geometric complexes associated with local systems $\loc \rightarrow X$ equivariant for a cyclic group action.  In Lemma \ref{numericalmeaning}, the Reidemeister torsion of these isotypic pieces is related to the equivariant Reidemeister torsion of the complexified local system $\loc_{\iota} \rightarrow X$ for embeddings $\iota: N \hookrightarrow \C.$  The local systems we consider will all be \emph{rationally acyclic}.

\begin{defn}
A local system $\loc \rightarrow X$ of $N$-vector spaces is \emph{acyclic} if $H^{*}(X, \loc) = 0.$  A local system $\loc' \rightarrow X$ of $O_N$-modules is \emph{rationally acyclic} (or $\mathbb{Q}$-acyclic or $N$-acyclic) if $\loc'_N,$ or equivalently $\loc_{\mathbb{Q}},$ is acyclic.
\end{defn}

\subsubsection{Reidemeister torsion of the cochain complex $\MS^{\bullet}(X, \loc)[\sigma - 1]$ for local systems of rationally acyclic $N$-vector spaces} \label{morsert}
Let $\loc \rightarrow \M$ be an equivariant local system of $N$-vector spaces acted on $N$-linearly by $\langle \sigma \rangle$ with $\sigma^p = 1.$  Assume that $\loc$ is rationally acyclic.  Let $A^{\bullet} = \MS^{\bullet}(X,\loc)$ for a vector field $X$ which satisfies Morse-Smale transversality and is weakly gradient-like with respect to a $\sigma$-invariant Morse function $f$ on $\M.$   By a \emph{volume form} on a local system $\loc \rightarrow \M$ of $N$-vector spaces, we mean a global section of $\det(\loc)^{*}.$  \medskip

Suppose we are given volume forms $\omega$ on $\loc \rightarrow \M, \omega_{\sigma - 1}$ on $\loc[\sigma - 1] \rightarrow \M_{\sigma},$ and $\omega_{P(\sigma)}$ on $\loc[P(\sigma)] \rightarrow \M_{\sigma}.$  These give rise to volume forms on the chain groups $A^{\bullet}[\sigma -1]$ and $A^{\bullet}[P(\sigma)].$  Indeed, let $\mathcal{C}$ be a set of representatives for the orbits of $\sigma$ acting on the critical points of $X.$  Then $A^{\bullet}[\sigma -1]$ has a ``geometric" $N$-basis given as follows. 

\begin{itemize}
\item
Suppose $x \in \mathcal{C}$ is not fixed by $\sigma.$  We let $\mathcal{O}_x := \bigoplus W(\sigma^j \cdot x) \otimes \loc_{\sigma^j \cdot x},$ where $\loc_x$ denotes the fiber of $\loc$ over $x.$ Then a basis for $\mathcal{O}_x[\sigma - 1]$ is given by $\{  \sum_j \sigma^j W^u(x) \otimes \sigma^j e^x_k \},$ where $e^x_1,...,e^x_r$ runs through a basis of $\loc_x.$     

\item
If $x \in \mathcal{C}$ is fixed by $\sigma,$ then let $f^x_1,...,f^x_s$ be a basis for $\loc_x[\sigma-1].$  A basis for $\mathcal{O}_x[\sigma - 1] = W^u(x) \otimes \loc_x[\sigma-1]$ is given by $\{ W^u(x) \otimes f^x_k \}.$ 
\end{itemize}     

Then a volume form $\theta^i_{\sigma - 1}$ is given by 

\begin{align*}
\theta^i_{\sigma - 1} \left( \bigwedge_{x \text{ not fixed}} \bigwedge_k \left( \sum_j W^u(\sigma^j x) \otimes \sigma^j e^x_k \right) \wedge \bigwedge_{x \text{ fixed}} \bigwedge_k W^u(x) \otimes f^x_k \right) &=& \prod_{x \text{ not fixed } } \omega(e^x_1 \wedge ... \wedge e^x_r ) \\  
&\times& \prod_{x \text{ fixed}} \omega_{\sigma - 1}(f^x_1 \wedge ... \wedge f^x_s),
\end{align*}

where $x$ ranges over all critical points of $\mathcal{C}$ of index $i.$  Using the volume form $\omega_{P(\sigma)},$ we can construct a collection of geometric volume forms $\theta^i_{P(\sigma)}$ in a completely analogous manner.  By the definition of Reidemeister torsion, it then follows that 

$$RT(A^{\bullet}[\sigma - 1], \theta_{\sigma - 1}), RT(A^{\bullet}[P(\sigma)], \theta_{P(\sigma)}) \in N^{\times}.$$

\subsection{Reidemeister torsion of the $\sigma$-isotypic pieces of $C^{\bullet}(\M,\loc;K)$}
\label{geomtwistedrtorsion}

\subsubsection{Reidemeister torsion of the cochain complex $C^{\bullet}(\M,\loc;K)[\sigma - 1]$ for local systems of rationally acyclic $N$-vector spaces} \label{cochainrt}
Let $\loc \rightarrow \M$ be an equivariant local system of $N$-vector spaces acted on $N$-linearly by $\langle \sigma \rangle$ with $\sigma^p = 1.$  Assume that $\loc$ is rationally acyclic.  Let $A^{\bullet} = C^{\bullet}(\M,\loc ;K)$ for a $\sigma$-equivariant triangulation of $\M.$ By a \emph{volume form} on a local system $\loc \rightarrow \M$ of $N$-vector spaces, we mean a global section of $\det(L)^{*}.$  \medskip

Suppose we are given volume forms $\omega$ on $\loc \rightarrow \M, \omega_{\sigma - 1}$ on $\loc[\sigma - 1] \rightarrow \M_{\sigma},$ and $\omega_{P(\sigma)}$ on $\loc[P(\sigma)] \rightarrow \M_{\sigma}.$  These give rise to volume forms on the chain groups $A^{\bullet}[\sigma -1]$ and $A^{\bullet}[P(\sigma)].$  Indeed, let $\mathcal{C} = \{ C \}$ be a set of representatives for the orbits of $\sigma$ acting on the $i$-cells of the triangulation $K.$  Then $A^{\bullet}[\sigma -1]$ has a ``geometric" $N$-basis given as follows. 

\begin{itemize}
\item
Suppose $C$ is not fixed by $\sigma.$  We let $\mathcal{O}_C := \bigoplus \sigma^j \cdot C \otimes \loc(C),$ where $\loc(C)$ denotes the $N$-vector space of sections of $\loc$ over $C.$  Then a basis for $\mathcal{O}_C[\sigma - 1]$ is given by $\{  \sum_j \sigma^j C \otimes \sigma^j e^C_k \},$ where $e^C_1,...,e^C_r$ runs through a basis of $\loc(C).$     

\item
If $C$ is fixed by $\sigma,$ then let $f^C_1,...,f^C_s$ be a basis for $L[\sigma-1](C).$  A basis for $\mathcal{O}_C[\sigma - 1] = C \otimes \loc[\sigma-1](C)$ is given by $\{ C \otimes f^C_k \}.$ 
\end{itemize}     

Then a volume form $\theta^i_{\sigma - 1}$ is given by 

\begin{align*}
\theta^i_{\sigma - 1} \left( \bigwedge_{C \text{ not fixed}} \bigwedge_k \left( \sum_j \sigma^j C \otimes \sigma^j e^C_k \right) \wedge \bigwedge_{C \text{ fixed}} \bigwedge_k C \otimes f^C_k \right) &=& \prod_{C \text{ not fixed } } \omega(e^C_1 \wedge ... \wedge e^C_r ) \\  
&\times& \prod_{C \text{ fixed}} \omega_{\sigma - 1}(f^C_1 \wedge ... \wedge f^C_s),
\end{align*}

Using the volume form $\omega_{P(\sigma)},$ we can construct a collection of geometric volume forms $\theta^i_{P(\sigma)}$ in a completely analogous manner.  By the definition of Reidemeister torsion, it then follows that 

$$RT(A^{\bullet}[\sigma - 1], \theta_{\sigma - 1}), RT(A^{\bullet}[P(\sigma)], \theta_{P(\sigma)}) \in N^{\times}.$$

\begin{rem}[the relative merits of the Morse-Smale complex]
$\S \ref{geomtwistedrtorsion}$ is nearly identical to $\S \ref{geomtwistedrtorsionms}.$  We see fit to mention the following:

\begin{itemize}
\item
The Bismut-Zhang theorem (see $\S \ref{statementofBZ}$) is expressed in the language of Morse-Smale complexes.  

\item
$\S \ref{geomtwistedrtorsionms}$ is a special case of $\S \ref{geomtwistedrtorsion}$ where we take the decomposition of $\M$ into (the closures of) unstable cells $W^u(x)$ as our cell decomposition.  

\item
The complex $\MS^{\bullet}(X,\loc)$ appears to be a slight enrichment of $C^{\bullet}(M, \loc; \{ W^u(x) \})$ in that each unstable manifold is endowed with a distinguished point which we were used to define the volume forms on $\MS(X, \loc)[\sigma - 1], \MS(X,  \loc)[P(\sigma)].$  However, as long as $\loc \rightarrow \M, \loc[\sigma-1] \rightarrow \M_{\sigma}$ and $\loc[P(\sigma)] \rightarrow \M_{\sigma}$ are unimodular, the aforementioned volume forms are ``independent of choice of basepoint". 
The Morse-Smale perspective proves more convenient in \cite{BZ}, \cite{BZ2} where the authors prove Cheeger-M\"{u}ller type theorems for local systems $\loc$ not assumed to be unimodular.    
\end{itemize}

\end{rem}

\subsubsection{Reidemeister torsion and norm compatibility}
\label{restrictionofscalars}

Let $N$ be a number field and $\pi: \mathrm{Spec} \; O_N \rightarrow \mathrm{Spec} \; \mathbb{Z}$ the natural map.  For any $O_N$ or $N$-module $P,$ we let $\pi_{*}P$ denote its restriction of scalars.  To describe the compatiblity between the Reidemeister torsion of the complex $P^{\bullet}$ and of the complex $\pi_{*} P^{\bullet}$ requires the notion of the norm of a projective $O_N$-module.  We refer to \cite[II, $\S 4$]{O} for the required foundations.  

\begin{lem}\label{oesterle}
Let $P$ be a projective $O_N$-module of finite rank $d.$  There is a canonical isomorphism 
$$\mathrm{det}_{\mathbb{Z}} (\pi_{*} P) = \mathrm{Norm}_{O_N/\mathbb{Z}}(\mathrm{det}_{O_N} (P)) \otimes (\mathrm{det}_{\mathbb{Z}}O_N)^{\otimes d}.$$  
\end{lem}

\begin{proof}
See \cite[II,$\S 4$]{O}.
\end{proof}

\begin{lem} \label{rtrestrictionofscalars}
Let $P^{\bullet}$ be a finite complex $N$-vector spaces with $P^i$ of dimension $d^i.$  Suppose $P^{\bullet}$ is equipped with a collection of volume forms $\omega_{\bullet}$ and that the cohomology spaces $H^{*}(P^{\bullet})$ are equipped with volume forms $\mu_{\bullet}.$  Suppose that $RT(P^{\bullet},\omega_{\bullet},\mu_{\bullet}) = f \in N^{\times}.$  Then $RT(\pi_{*}P^{\bullet}, \mathrm{Norm}_{N/\mathbb{Q}} \omega_{\bullet}, \mathrm{Norm}_{N/\mathbb{Q}} \mu_{\bullet}) = \mathrm{Norm}_{N/\mathbb{Q}} f.$  
\end{lem}

\begin{rem}
In accordance with Lemma \ref{oesterle}, $\mathrm{Norm}_{N/\mathbb{Q}} \omega_{\bullet}, \mathrm{Norm}_{N/\mathbb{Q}} \mu_{\bullet}$ are not volume forms, but rather $\mathrm{Norm}_{N/\mathbb{Q}} \omega_{\bullet} \otimes \gamma_0^{\otimes \dim_N P^{\bullet}}$ and $\mathrm{Norm}_{N/\mathbb{Q}} \mu_{\bullet} \otimes \gamma_0^{\otimes \dim_N H^{*}(P^{\bullet})}$ are, for  $\gamma_0 \in (\wedge^{\mathrm{top}}_{\mathbb{Q}}N)^{*}.$  However, as will be clear from the proof to follow, the Reidemeiester torsion of the resulting complex is independent of choice of $\gamma_0$; this is why $\gamma_0$ is omitted from the notation.
\end{rem}

\begin{proof}
Let $d_i = \dim_N P^i,h_i = \dim_N H^i(P^{\bullet}).$  By Lemma \ref{oesterle},
\begin{eqnarray*}
&{}& \mathrm{det}_{\mathbb{Q}}(\pi_{*} P^{\bullet}) \otimes \mathrm{det}_{\mathbb{Q}} (\pi_{*} H^{*}(P^{\bullet}))^{-1} \\
&=& \mathrm{Norm}_{N/\mathbb{Q}} (\mathrm{det}_N P^{\bullet}) \otimes \mathrm{Norm}_{N/\mathbb{Q}}(\mathrm{det}_N H^{*}(P^{\bullet}))^{-1} \otimes (\mathrm{det}_{\mathbb{Q}} N)^{\otimes \sum {}^{*} d_i - \sum {}^{*} h_i}  \\
&=&  \mathrm{Norm}_{N/\mathbb{Q}} (\mathrm{det}_N P^{\bullet}) \otimes \mathrm{Norm}_{N/\mathbb{Q}}(\mathrm{det}_N H^{*}(P^{\bullet}))^{-1}. 
\end{eqnarray*} 
The last equality follows because $\sum {}^{*} d_i = \sum {}^{*} h_i,$ both equalling the Euler characteristic of the complex $P^{\bullet}.$  The complex $P^{\bullet}$ gives rise to the section $s_{P^{\bullet}} = f \cdot \omega_{\bullet} \otimes \mu_{\bullet}^{-1}$ of $\mathrm{det}_{N}(P^{\bullet}) \otimes \mathrm{det}_{N} (H^{*}(P^{\bullet}))^{-1},$ where $f = RT(P^{\bullet},\omega_{\bullet},\mu_{\bullet}).$  Therefore,
\begin{eqnarray*}
\mathrm{Norm}_{N/\mathbb{Q}} s_{P^{\bullet}} &=& \mathrm{Norm}_{N/\mathbb{Q}} (f \cdot \omega_{\bullet} \otimes \mu_{\bullet}^{-1}) \\
&=& \mathrm{Norm}_{N / \mathbb{Q}} f \cdot (\mathrm{Norm}_{N / \mathbb{Q}} \omega_{\bullet}) \otimes (\mathrm{Norm}_{N / \mathbb{Q}} \mu_{\bullet})^{-1}.  
\end{eqnarray*} 
Unravelling the isomorphism from Lemma \ref{oesterle} (see \cite[II, $\S 4.2$]{O}), we readily find that $\mathrm{Norm}_{N/\mathbb{Q}} s_{P^{\bullet}} = s_{\pi_{*} P^{\bullet}}.$  The lemma follows.     
\end{proof}

\subsubsection{Applying norm compatibility to the geometric complexes $C^{\bullet}(\M,\loc;K), \MS(X, \loc)$ for local systems of $O_N$-modules}
\label{rtorsionaltprod}

Suppose $\loc \rightarrow \M$ is an equivariant local system of projective $O_N$-modules for which $\det(\loc)^{*}$ is trivial and $\det(\loc[\sigma-1])^{*}, \det(\loc[P(\sigma)])^{*} \rightarrow \M_{\sigma}$ are trivial with bases $\omega, \omega_{\sigma - 1}, \omega_{P(\sigma)}.$  Suppose further that $\loc_N$ is acyclic.  Let $A^{\bullet} = C^{\bullet}(\M,\loc ;K)$ or $\MS(X,\loc),$ as defined in $\S \ref{geomtwistedrtorsionms}$ and $\S \ref{geomtwistedrtorsion}.$  As before, we construct volume forms $\theta^i_{\sigma-1}, \theta^i_{P(\sigma)-1}$ now generators of the \emph{free $O_N$-modules} $\left( \wedge_{O_N}^{\mathrm{top}} A^i[\sigma - 1] \right)^{*}$  and $\left( \wedge_{O_N}^{\mathrm{top}} A^i[P(\sigma)] \right)^{*}$ respectively.  Let $\gamma_0 \in \left(\wedge_{\mathbb{Z}}^{\mathrm{top}} O_N \right)^{*}.$

As a corollary to Lemma \ref{rtrestrictionofscalars}, we obtain the main result of this section:
\begin{cor} \label{altprodofcohomology}
Let $A^{\bullet} = C^{\bullet}(\M,\loc ;K)$ or $\MS(X,\loc).$  If $RT(A^{\bullet}[\sigma - 1], \theta_{\sigma-1}) = f, RT(A^{\bullet}[P(\sigma)], \theta_{P(\sigma)}) = f' \in N^{\times},$ then 
$$\mathrm{Norm}_{N/ \mathbb{Q}} f =  \pm \prod {}^{*} |H^i(A^{\bullet}[\sigma - 1])|, \mathrm{Norm}_{N/ \mathbb{Q}} f' =  \pm \prod {}^{*} |H^i(A^{\bullet}[P(\sigma)])|.$$
\end{cor}

\begin{proof}
As explained in item $(1)$ of $\S \ref{examples},$ the Reidemeister torsion for any finite, rationally acyclic complex $C^{\bullet}$ of $\mathbb{Z}$-modules with volume forms $\alpha_i$ given by a generator of $(\wedge_{\mathbb{Z}}^{\mathrm{top}} C^i)^{*}$ satisfies

$$RT(C^{\bullet}, \alpha) = \pm \prod {}^{*} |H^i(C^{\bullet})|.$$

The result follows immediately by Lemma \ref{rtrestrictionofscalars}, with $C^{\bullet} = \MS^{\bullet}(X, \loc)[\sigma - 1], \MS^{\bullet}(X, \loc)[P(\sigma)].$  
\end{proof}

%
%
%

\subsubsection{Twisted Reidemeister torsion of $\MS^{\bullet}(X,\loc), C^{\bullet}(\M,\loc;K)$ and compatibility with base change to $\C$}
\label{complexifiedcomplex}
Let $N$ be a number field.  Let $\iota: N \hookrightarrow \C$ be a complex embedding.  Let $\loc \rightarrow M$ be an equivariant local system of $O_N$-modules acted on by $\langle \sigma \rangle$ with $\sigma^p = 1.$  Suppose that there are global volume forms $\omega$ on $\det(\loc) \rightarrow \M, \omega_{\sigma - 1}$ on $\loc[\sigma - 1] \rightarrow \M_{\sigma},$ and $\omega_{P(\sigma)}$ on $\loc[P(\sigma)] \rightarrow \M_{\sigma}.$  \medskip

Suppose further that there is a metric $h$ on $\loc_{\iota} := \loc \otimes_{\iota} \C$ which induces the norm $|\omega_{\iota}|$ and which also induces the norms $|\omega_{\sigma-1,\iota}|, |\omega_{P(\sigma),\iota}|$ on $\loc[\sigma-1] \otimes_{\iota} \C \rightarrow \M_{\sigma}, \loc[P(\sigma)] \otimes_{\iota} \C \rightarrow \M_{\sigma}.$  Suppose further that $\sigma$ acts isometrically on $\loc \rightarrow \M$ with respect to $h.$  

\begin{lem} \label{numericalmeaning} 
Let $A^{\bullet} = C^{\bullet}(\M,\loc ;K)$ or $\MS(X,\loc).$  Suppose $RT(A^{\bullet}[\sigma - 1], \theta_{\sigma-1}) = f, RT(A^{\bullet}[P(\sigma)], \theta_{P(\sigma)}) = f' \in N^{\times}.$  Then
$$\log RT_{\sigma}(\M,\loc_{\iota},h) = \log |\iota(f)| - \frac{1}{p-1} \log |\iota(f')|.$$
Furthermore, if the volume forms $\omega_{\sigma-1}, \omega_{P(\sigma)}$ giving rise to $\theta_{\sigma -1}, \theta_{P(\sigma)}$ are generators of $\mathrm{det}_{O_N}(\loc[\sigma - 1])^{*}$ and $\mathrm{det}_{O_N}(\loc_{P(\sigma)})^{*},$ then 

$$|\mathrm{Norm}_{N/\mathbb{Q}}(f)| = \prod {}^{*} | H^i(A^{\bullet}[\sigma - 1])|,$$ 
$$|\mathrm{Norm}_{N /\mathbb{Q}}(f')| = \prod {}^{*} | H^i(A^{\bullet}[P(\sigma)])| .$$ 
\end{lem}

\begin{proof}
For any bounded chain complex of finite dimensional $N$-vector spaces $B^{\bullet},$ volume forms $\theta$ on $B^{\bullet}$ give rise to volume forms $\theta_{\iota}$ on the base changed complex $B^{\bullet}_{\iota}.$  According to the recipe descirbed after Definition \ref{rtvolumeforms} for computing Reidemeister torsion, if $g$ is the Reidemeister torsion of $(B^{\bullet}, \theta),$ then $\iota(g)$ is the Reidemeister torsion of $(B^{\bullet}_{\iota},\theta_{\iota}).$  By the hypothesis that the norms $|\omega_{\iota}|, |\omega_{\sigma-1,\iota}|,$ and $|\omega_{P(\sigma),\iota}|$ are all induced by the metric $h,$ the first part of the proposition follows from this observation.
The second part follows by Corollary \ref{altprodofcohomology}.
\end{proof}

%
%

\subsection{Statements of two variants of the Cheeger-M\"{u}ller theorem}
\label{luckcm}

We recall the statements of one variant of the untwisted Cheeger-M\"{u}ller theorem and one of its twisted counterpart.  These are one crucial input into the main theorem \cite[Theorem 7.4]{Lip2}.

\begin{thm}[\cite{Mu2}] \label{muller}
Let $M$ be a closed manifold and  $L \rightarrow M$ a metrized, unimodular local system of $\mathbb{C}$-vector spaces.  Then
$$\tau(M,L) = RT(M,L).$$ 
\end{thm}

\begin{thm}[\cite{Lu}, theorem 4.5] \label{luck}
Let $\M$ be a closed manifold and  $\loc \rightarrow \M$ a $\langle \sigma \rangle$-equivariant local system of $\mathbb{C}$-vector spaces.  Suppose that $\loc$ is equipped with a metric with respect to which it is \emph{unitarily flat} and for which a finite cyclic group $\langle \sigma \rangle$ acts equivariantly by isometries.  Then
$$\tau_{\sigma}(\M, \loc) = RT_{\sigma}(\M,\loc).$$ 
\end{thm}

See Definition \ref{defntwistedrt} and the subsequent remarks for the definition of twisted and untwisted analytic torsion, respectively denoted $\tau_{\sigma}$ and $\tau.$  See Definition \ref{morsesmalert} for the definition of twisted Reidemeister torsion, denoted $RT_{\sigma},$ which defines Reidemeister torsion, denoted $RT,$ when $\langle \sigma \rangle = 1.$

\begin{rem}[History of the Cheeger-M\"{u}ller theorem]
The remarkable discovery that the untwisted Cheeger-M\"{u}ller theorem might be true was made by Ray and Singer in \cite{RS}.  The untwisted Cheeger-M\"{u}ller theorem was proven independently by Cheeger in \cite{Ch} and M\"{u}ller in \cite{Mu1} for orthogonal local systems.  M\"{u}ller later generalized this result to arbitrary unimodular local systems in \cite{Mu2}.  A more general and difficult to state variant for arbitrary local systems was proven by Bismut-Zhang in \cite{BZ}.

The twisted Cheeger-M\"{u}ller theorem for the trivial local system was first proved by Lott-Rothenberg in \cite{LR}.  It was generalized to equivariant orthogonal local systems by L\"{u}ck in \cite{Lu}.  The ultimate version we will use was proven by Bismut-Zhang in \cite{BZ2}; we apply Bismut-Zhang's variant of the twisted Cheeger-M\"{u}ller thoerem to equivariant unimodular local sytems.  We discuss the Bismut-Zhang variant at length in $\S \ref{dRecmt}.$   
\end{rem}

\section[Cheeger-M\"{u}ller for finite chain complexes]{Equivariant Cheeger-M\"{u}ller theorem for finite chain complexes}
\label{finitechaincomplex}

The equivariant Cheeger-M\"{u}ller theorem \ref{luck} will be crucial in the sequel.  To motivate it, we derive a version for finite dimensional chain complexes. 

\begin{itemize}
\item
In $\S \ref{untwistedcheegermuller},$ we recall the statement of the untwisted Cheeger-M\"{u}ller theorem for finite chain complexes. 

\item
In $\S \ref{chaincomplextwistedatorsion},$ we define the twisted analytic torsion of a chain complex $A^{\bullet}$ of free abelian groups for which $A^{\bullet}_{\R}$ is metrized and acted on isometrically by a finite group $\Gamma.$

\item
In $\S \ref{guessing},$ we derive a homological expression for the twisted analytic torsion of a complex $A^{\bullet}$ as above.
\end{itemize}

\subsection{Statement of the untwisted Cheeger-M\"{u}ller theorem for finite chain complexes}
\label{untwistedcheegermuller}

Let $(A^{\bullet},d)$ be a chain complex of finite-free abelian groups, equipped with metrics on $A^{\bullet}_{\mathbb{R}},$ such that $\vol(A^i_{\mathbb{R}} / A^i) = 1$ for each $i.$  This complex has a Laplace operator 

$$\Delta_j = d_{j-1} d_{j-1}^{*} + d_j^{*} d_j.$$

Associated to $\Delta_j$ are spectral generating functions

$$\zeta_{j,A^{\bullet}}(s) = \sum_{\lambda \text{ eigenvalue of } \Delta_j} \lambda^{-s}, Z_{A^{\bullet}}(s) = \frac{1}{2} \sum (-1)^j j \zeta_{j, A^{\bullet}}(s).$$

\begin{defn}
The \emph{(untwisted) analytic torsion} of the metrized complex $A^{\bullet}$ is defined to be 
\begin{equation}
\tau(A^{\bullet}) = \exp( - Z'_{A^{\bullet}}(0) ) 
\end{equation}
\end{defn}

There is a Hodge decomposition for any finite chain complex which allows us to represent each cohomology class of $A^{\bullet}_{\mathbb{R}}$ uniquely by a \emph{harmonic cochain}, defined to be an element of the kernel of $\Delta_{\bullet}.$  Thus, $H^i(A^{\bullet}_{\mathbb{R}})$ inherits a metric from the space of harmonic cochains.  Define the regulators to be 
$$R^i(A^{\bullet}) = \vol(H^i(A^{\bullet}_{\mathbb{R}}) / \mathrm{im} \{ H^i(A^{\bullet}) \rightarrow H^i(A^{\bullet}_\R) \} ) .$$  
\begin{lem} 
There is an equality
\begin{equation} \label{untwistedcomp}
\sum {}^{*} \log R^i(A^{\bullet}) - \log | H^i(A^{\bullet})_{\mathrm{tors}} | = \log \tau(A^{\bullet}).
\end{equation}
\end{lem}

\begin{proof}
See \cite[$\S 2$]{BV} or \cite[$\S 1$]{Ch}.
\end{proof}

\begin{rem}
Suppose that the complex $A^{\bullet}$ arises as the group of $L$-valued cochains - associated to a particular triangulation - for a metrized local system $L \rightarrow M$ over a compact, smooth, Riemannian manifold $M.$  The content of the de Rham version of untwisted Cheeger-M\"{u}ller theorem (see \cite{Mu1}) is that the identity (\ref{untwistedcomp}) ``passes to a de Rham limit" under successively finer triangulations of $M.$  
\end{rem}

\subsection{Definition of twisted analytic torsion for finite chain complexes}
\label{chaincomplextwistedatorsion}

Let $A^{\bullet}$ be a complex of metrized, finite free $\mathbb{Z}$-modules acted on isometrically by a finite group $\Gamma.$  We can form the equivariant zeta functions
$$\zeta_{j,A^{\bullet}, \Gamma}(s) = \sum_{\lambda \text{ eigenvalue of } \Delta_j} \lambda^{-s} [E_{\lambda}] \in \mathrm{Rep}(\Gamma)_{\mathbb{C}},$$
where $[E_{\lambda}]$ is the $\lambda$-eigenspace of the combinatorial Laplacian $\Delta_j,$ thought of as a representation of $\Gamma.$ 

\begin{defn}
The \emph{equivariant (or twisted) analytic torsion} $\tau_{\Gamma}(A^{\bullet})$ of the finite, metrized chain complex $A^{\bullet},$ acted on isometrically by a finite group $\Gamma,$ is defined to be
\begin{equation}
\tau_{\Gamma}(A^{\bullet}) := Z_{A^{\bullet}, \Gamma}'(0), \text{ where } Z_{A^{\bullet}, \Gamma}(s) = \frac{1}{2} \sum (-1)^j j \zeta_{j, A^{\bullet}, \Gamma},
\end{equation}
As a notational shorthand, we define
\begin{equation}
\tau_{\sigma}(A^{\bullet}) := \tr \{ \sigma | \tau_{\Gamma}(A^{\bullet}) \}.
\end{equation}
\end{defn}

\subsection{A homological expression for the twisted analytic torsion $\tau_{\sigma}(A^{\bullet})$}
\label{guessing}

Suppose the metrized complex $A^{\bullet}$ from above is endowed with an involution $\sigma$ acting by isometries.  Of course, we are more generally interested in $\sigma^p = 1,$ but all difficulties are already present for $p = 2,$ the necessary changes are clear, and the notation is simpler. 

We let $D^i = A^{i,+} \oplus A^{i,-}.$  Unfortunately, it is usually not the case that $D^i = A^i.$

\begin{lem} \label{twistedfinitecomplex}
Assume that $\vol(A^i) = 1$ for every $i.$  Then letting $A'^{\bullet} := A^{\bullet} / \left( A^{\bullet}[\sigma - 1] \oplus A^{\bullet}[\sigma + 1] \right),$
\begin{equation} \label{twistedguess}
\tau_{\sigma}(A^{\bullet}) = \frac{\prod{}^{*} |H^i(A^{\bullet}[\sigma - 1])_{\mathrm{tors}}|^{-1} \prod{}^{*} R(A^i[\sigma - 1])}{ \prod{}^{*} |H^i(A^{\bullet}[\sigma + 1])_{\mathrm{tors}}|^{-1} \prod{}^{*} R(A^i[\sigma + 1])  } \times  \prod{}^{*} | H^i(A'^{\bullet})|^{-1}. 
\end{equation}
\end{lem}

\begin{proof}
To derive a combinatorial expression for $\tau_{\sigma}(A^{\bullet}),$ we split the complex into the $\pm 1$ eigenspaces of $\sigma.$  Let $A^{\bullet, +} = (A^{\bullet})^{\sigma - 1}$ and $A^{\bullet,-} = (A^{\bullet})^{\sigma + 1}.$  Note that
\begin{eqnarray*}
\zeta_{j,\sigma}(s) &=& \sum \tr\{ \sigma | A^i(\lambda) \} \lambda^{-s}\\
&=& \sum \dim A^{i,+}(\lambda) \lambda^{-s} - \sum \dim A^{i,-}(\lambda) \lambda^{-s} \\
&=& \zeta_{j, A^{+,\bullet}}(s) - \zeta_{j, A^{-,\bullet}}(s).
\end{eqnarray*}
So by the computation of \cite[$\S 2$]{BV} of untwisted analytic torsion for finite chain complexes, namely equation $(\ref{untwistedcomp}),$ we arrive at the identity
\begin{equation} \label{twistedsimplification}
\tau_{\sigma}(A^{\bullet}) = \frac{\prod{}^{*} |H^i(A^{+,\bullet})_{\mathrm{tors}}|^{-1} \prod{}^{*} R(A^{+,i})}{ \prod{}^{*} |H^i(A^{-,\bullet})_{\mathrm{tors}}|^{-1} \prod{}^{*} R(A^{-,i})  } \times \prod{}^{*} \vol(D^i)^{-1}.
\end{equation}
Because $\vol(A^i) = 1,$
$$\vol(D^i) = [A^i : D^i].$$
This allows us to re-express the product of volumes homologically:
\begin{align} \label{homologicalexpressionforvolume}
\prod {}^{*}\vol(D_i)^{-1} &= \prod{}^{*} [A^i : D^i]^{-1} \nonumber \\
&= \prod{}^{*} | H^i(A'^{\bullet})|^{-1}.
\end{align}
Equation \eqref{homologicalexpressionforvolume} makes sense because $A^i_{\mathbb{Q}} = D^i_{\mathbb{Q}}$ and $H^i(A'^{\bullet})$ is finite for every $i.$  Substituting \eqref{homologicalexpressionforvolume} back into equation $(\ref{twistedsimplification})$ proves the lemma.   
\end{proof}

\section[An estimate of twisted Reidemeister torsion]{A ``triangulation independent" estimate of twisted Reidemeister torsion}
\label{preliminaryestimate}

In Lemma $\ref{twistedfinitecomplex},$ we derived a homological expression for the twisted analytic torsion of a finite metrized complex $A^{\bullet}$ of free abelian groups acted on isometrically by $\mathbb{Z} / 2\mathbb{Z} = \langle \sigma \rangle.$   This computation specializes to the following when $A^{\bullet}$ is $\mathbb{Q}$-acyclic: 

\begin{equation} \label{acyclictwistedguess}
\tau_{\sigma}(A^{\bullet}) = \frac{\prod{}^{*} |H^i(A^{+,\bullet})_{\mathrm{tors}}|^{-1}  }{ \prod{}^{*} |H^i(A^{-,\bullet})_{\mathrm{tors}}|^{-1}  } \times  \prod{}^{*} | H^i(A'^{\bullet})|^{-1}, 
\end{equation}

where we recall that $A'^{\bullet} = A^{\bullet} / \left(A^{+,\bullet} \oplus A^{-, \bullet} \right).$  The expression on the right side of $(\ref{acyclictwistedguess})$ will come up often enough that we see fit to give it a definition.

\begin{defn} \label{nrt}
Let $C$ be any bounded complex of torsion-free $\mathbb{Z}[\sigma]$-modules, where $\sigma^p = 1$ for some prime $p,$ and assume $C_{\mathbb{Q}}$ is acyclic.  We define the \emph{naive equivariant Reidemeister torsion of $C$}, denoted $NRT_{\sigma},$ as  
\begin{eqnarray*}
\log NRT_{\sigma}(C^{\bullet}) &=& \left\{ \sum {}^{*} \log |H^i(C^{\sigma - 1})| - \frac{1}{p-1} \log |H^i(C^{P(\sigma)})| \right\} + \left\{ \sum {}^{*} \log | C'_i | \right\} \\
&=&  \left\{ \sum {}^{*} \log |H^i(C^{\sigma - 1})| - \frac{1}{p-1} \log |H^i(C^{P(\sigma)})| \right\} + \left\{ \sum {}^{*} \log | H^i(C') | \right\},
\end{eqnarray*}  
where $C' := C / (C^{\sigma - 1} \oplus C^{P(\sigma)}).$
\end{defn}

We use the word ``naive" because $NRT_{\sigma}(A^{\bullet})$ does not always equal $RT_{\sigma}(A^{\bullet}).$  Nonetheless, it is concrete and we can directly relate it to $RT_{\sigma}(A^{\bullet})$ in geometric situations where $A^{\bullet}$ arises as the group of $\loc$-valued cochains on $\M$ for some $\langle \sigma \rangle$-equivariant local system $\loc \rightarrow \M$ or as the Morse-Smale complex of $\loc$ for some gradient vector field $X$ on $\M.$ \bigskip

The goal of this section is to prove that approximately   
\begin{equation} \label{wanted}
NRT_{\sigma}(A^{\bullet}) \sim  \prod{}^{*} \frac{| H^i(A^{\bullet})_{\mathrm{tors}}^{\sigma - 1}   |}{|  H^i(A^{\bullet})_{\mathrm{tors}}^{\sigma +1}     |} \text{ for } \sigma^2 = 1 \text{ (and an analogue for } \sigma^p = 1),
\end{equation}  
where $\sim$ denotes equality up to a ``controlled" power of 2.  On the right side of (\ref{wanted}), $\sigma$ acts on the cohomology groups of $A^{\bullet},$ \emph{not} on $A^{\bullet}$ itself as in the definition of naive equivariant Reidemeister torsion.  Equation (\ref{wanted}) will be important for two reasons:  
\begin{itemize}
\item[-]
The twisted Cheeger-M\"{u}ller theorem relates analytic torsion to Reidemeister torsion, \emph{not directly} to sizes of cohomology groups.  In the interest of proving cohomology growth theorems following \cite{BV}, it is important to concretely understand the relationship between the twisted Reidemeister torsion $RT_{\sigma}(\M, \loc)$ and the cohomology $H^{*}(\M, \loc).$

\item[-]
The numerical torsion functoriality theorems proven in \cite[\S 7.2]{Lip2} 
concern $RT_{\sigma}(\M,\loc).$  In order to prove numerical versions of torsion functoriality, in the spirit of \cite{CV}, we need to directly relate $RT_{\sigma}$ to sizes of cohomology groups.  
\end{itemize}

To relate $RT_{\sigma}(A^{\bullet})$ to the sizes of cohomology groups $H^{*}(A^{\bullet}),$ 

\begin{itemize}
\item
In $\S \ref{fixedprime},$ we separate the expression for $NRT_{\sigma}$ into a 2-power torsion part and prime to 2 torsion parts.

\item
In $\S \ref{geomestimates},$ we use a spectral sequence argument to relate $NRT_{\sigma}(A^{\bullet})$ to $\prod{}^{*} \frac{| H^i(A^{\bullet})_{\mathrm{tors}}^{\sigma - 1}   |}{|  H^i(A^{\bullet})_{\mathrm{tors}}^{\sigma +1}     |}$ in the case where $A^{\bullet}$ is the complex of $\loc$-valued cochains of $\M$ for an equivariant local system $\loc \rightarrow \M$ of free abelian groups adapted to an equivariant triangulation. 


\item
In $\S \ref{smalldifference},$ we will compare the naive twisted Reidemeister torsion of the cochain complex $C^{\bullet}(\M,\loc;K),$ acted on isometrically by $\mathbb{Z} / p\mathbb{Z},$ to its actual twisted Reidemeister torsion.  We obtain a homological expression for the difference between these two quantities, which is often zero.  This is made precise in Proposition \ref{rtvsnrt}, the main result of $\S \ref{preliminaryestimate}.$     
\end{itemize}

\subsection{Adapting $NRT_{\sigma}(A^{\bullet})$ to the prime 2}
\label{fixedprime}

We decompose the product $(\ref{acyclictwistedguess})$ defining $NRT_{\sigma}(A^{\bullet})$ as follows:
\begin{eqnarray}
\log NRT_{\sigma}(A^{\bullet}) &=& - \left( \sum{}^{*} \log \left|H^i(A^{\bullet})[2^{-1}]^{\sigma - 1} \right|  - \log \left|H^i(A^{\bullet})[2^{-1}]^{\sigma + 1} \right| \right) \label{awayfrom2} \\
&-& \left(      \sum{}^{*} \log \left|H^i(A^{+,\bullet})[2^{\infty}] \right|  - \log \left|H^i(A^{-,\bullet})[2^{\infty}]\right|     \right) \label{at2} \\
&-& \left( \sum{}^{*} \log \left| H^i(A'^{\bullet})\right| \right) \label{fixedpts},
\end{eqnarray}
where $A'^{\bullet} := A^{\bullet} / \left( A^{\bullet}[\sigma - 1] \oplus A^{\bullet}[\sigma+1] \right).$

\begin{lem} \label{everythingrearranged}
\begin{equation} \label{firstestimate}
|(\ref{at2})| + |(\ref{fixedpts})| \leq \log \left|H^{*}(A^{\bullet})[2^{\infty}] \right| +  2 \log \left|H^{*}(A'^{\bullet}) \right|
\end{equation} 
\end{lem}

\begin{proof}
We first bound $(\ref{at2})$:
\begin{eqnarray*}
|(\ref{at2})| &\leq&   \sum \log \left|H^i(A^{+,\bullet})[2^{\infty}] \right|  + \log \left|H^i(A^{-,\bullet})[2^{\infty}]\right| \\
&=&  \log \left|H^{*}(A^{\bullet}[\sigma - 1] \oplus A^{\bullet}[\sigma + 1])[2^{\infty}] \right| 
\end{eqnarray*}
By using the long exact sequence associated to $0 \rightarrow A^{\bullet}[\sigma - 1] \oplus A^{\bullet}[\sigma + 1] \rightarrow A^{\bullet} \rightarrow A'^{\bullet} \rightarrow 0,$ we obtain the further bound
$$|(\ref{at2})| \leq  \log \left|H^{*}(A^{\bullet})[2^{\infty}] \right| +  \log \left|H^{*}(A'^{\bullet}) \right|.$$
The lemma follows.
\end{proof}

\subsection{Further estimates when $A^{\bullet}$ arises geometrically}
\label{geomestimates}

A key observation is that when $A^{\bullet}$ arises as the group of cochains of a triangulation or the Morse-Smale complex for a gradient vector field, then $H^i(A'^{\bullet})$ (the $'$ notation is defined in $\S \ref{commonnotation1}$) has a nice homological interpretation.  In this section, we give an interpretation to the second summand $H^i(A'^{\bullet})$ occurring on the right side of equation $(\ref{firstestimate})$ from lemma $\ref{twistedfinitecomplex}$ in the following geometric situation: $A^{\bullet}$ is the complex of $\loc$-valued cochains on $\M$ corresponding to some equivariant triangulation of $\M$ and some equivariant, metrized local system $\loc \rightarrow \M$ of free abelian groups.  \bigskip

Suppose that the compact manifold $\M$ is acted on isometrically by $\mathbb{Z}/2\mathbb{Z} = \langle \sigma \rangle.$  Let $\loc \rightarrow \M$ be an equivariant, metrized local system.  Assume that $\loc$ is $\mathbb{Q}$-acyclic.  Let $A^{\bullet}$ be the group of cochains arising from an equivariant triangulation of $\M,$ one which extends a triangulation on the fixed point set $\M_{\sigma} = M.$ \bigskip

\begin{prop} \label{geominterpretation}
We can identify
$$H^{*}(A'^{\bullet}) = H_c^{*}((\mathcal{M} - M)/ \langle \sigma \rangle, \mathcal{L}_{\mathbb{F}_2}),$$
where, abusing notation, $\loc$ denotes the unique descent of the local system $\loc|_{\M - M}$ to the quotient space $(\M - M) / \langle \sigma \rangle.$ 
\end{prop}

\begin{proof}
Let $D^{\bullet} = A^{\bullet}[\sigma - 1] \oplus A^{\bullet}[\sigma +1]$ so that $A'^{\bullet} = A^{\bullet} / D^{\bullet}.$  Let $\overline{D}^{\bullet} = \im(D^{\bullet} \rightarrow A^{\bullet}_{\mathbb{F}_2}).$  We can identify $\overline{D}^{\bullet} = A^{\bullet}_{\mathbb{F}_2}[\sigma - 1] + \mathrm{im} (\sigma - 1).$  Therefore, 
\begin{equation}
A'^{\bullet} = \sigma \text{-coinvariants of } A^{\bullet}_{\mathbb{F}_2} / A^{\bullet}_{\mathbb{F}_2}[\sigma - 1].
\end{equation}
There is an exact sequence of complexes
$$0 \rightarrow A^{\bullet}_{\mathbb{F}_2}[\sigma - 1] \rightarrow A^{\bullet}_{\mathbb{F}_2} \rightarrow  A^{\bullet}_{\mathbb{F}_2} / A^{\bullet}_{\mathbb{F}_2}[\sigma - 1] \rightarrow 0.$$
Because $\sigma$ acts freely on the rightmost term, taking coinvariants is exact.  Therefore, we can identify coinvariants of the rightmost term with the quotient of coinvariant groups
\begin{equation} \label{coinvariants}
A'^{\bullet} \cong A^{\bullet}_{\mathbb{F}_2, \langle \sigma \rangle} / (A^{\bullet}_{\mathbb{F}_2}[\sigma - 1])_{\langle \sigma \rangle}.
\end{equation}
Since taking coinvariants has the effect of identifying $\sigma$-equivalent cells, equation (\ref{coinvariants}) implies that
\begin{equation} \label{relativecoh}
H^{*}(A'^{\bullet}) \cong H^*(\M/ \langle \sigma \rangle, M, \loc_{\mathbb{F}_2}),
\end{equation}
where $\M / \langle \sigma \rangle$ denotes the point set topological quotient.  Since the fixed point set $M \subset \M / \langle \sigma \rangle$ is a Euclidean neighborhood retract, the right side of $(\ref{relativecoh})$ can be identified with
\begin{equation} \label{cptcoh}
H^{*}_c(\M / \langle \sigma \rangle - M, \loc_{\mathbb{F}_2}) = H^{*}_c( (\M - M) / \langle \sigma \rangle, \loc_{\mathbb{F}_2}),
\end{equation} 
as desired.
\end{proof}

Next, we carry through a spectral sequence argument to estimate the cohomology $H^{*}_c((\M - M)/ \langle \sigma \rangle, \loc_{\mathbb{F}_2})$ in terms of $H^{*}(M, \loc_{\mathbb{F}_2})$ and $H^{*}(\M, \loc_{\mathbb{F}_2}).$

\begin{prop} \label{finitecomplexestimate}
\begin{eqnarray} \label{thirdest}
&{}& \log NRT_{\sigma}(\loc) \notag \\  
&=& - \textcolor{blue}{\sum {}^\ast} \left(\log \left|H^i(\M, \loc)[2^{-1}]^{\sigma - 1} \right|  - \log \left|H^i(\M, \loc)[2^{-1}]^{\sigma + 1} \right| \right) \notag \\  
&+& O\left( \log|H^{*}(\mathcal{M}, \mathcal{L})[2^{\infty}]| +\log |H^{*}(\mathcal{M}, \loc_{\mathbb{F}_2})| + \log |H^{*}(M, \loc_{\mathbb{F}_2})| \right).  
\end{eqnarray}
\end{prop}

\begin{proof}
By the long exact sequence for cohomology, relative to the pair $(\mathcal{M}, M),$ we obtain that
\begin{equation} \label{firstest}
\log |H^{*}_c(\mathcal{M} - M, \loc_{\mathbb{F}_2})| \leq \log |H^{*}(\mathcal{M}, \loc_{\mathbb{F}_2})| +  \log |H^{*}(M, \loc_{\mathbb{F}_2})|. 
\end{equation}
We can relate the left side of (\ref{firstest}) to the cohomology of the quotient using a spectral sequence argument.  Consider the fibration
$$\begin{CD}
\mathcal{M} - M @>>> (\mathcal{M} - M) \times_{\sigma} E \langle \sigma \rangle \\
@. @V{\pi}VV \\
@. B \langle \sigma \rangle
\end{CD}$$
together with the sheaf $\loc_{\mathbb{F}_2} \rightarrow (\mathcal{M} - M) / \langle \sigma \rangle.$  The corresponding Serre spectral sequence has $E_2$ page
$$E_2^{p,q} = H^p( \langle \sigma \rangle, H^q(\mathcal{M} - M, \loc_{\mathbb{F}_2})) \implies H^{p+q}((\mathcal{M} - M) / \langle \sigma \rangle, \loc_{\mathbb{F}_2}).$$
The entries on the $E_2$ page are all finite abelian groups.  As we turn the page, the orders of the groups appearing can only decrease.  Thus, 
$$\log  | H^{*}((\mathcal{M} - M) / \langle \sigma \rangle, \loc_{\mathbb{F}_2}) | \leq \sum_{p,q \leq \dim \mathcal{M}} \log | E_2^{p,q}|.$$
But fortunately, the cohomology of cyclic groups is well-understood.  In particular, 
$$E_2^{p,q} = \begin{cases}
\text{ quotient of } H^q(\mathcal{M} - M, \loc_{\mathbb{F}_2})^{\sigma} & \text{ if } p \text{ is even }\\
 H^1(\langle \sigma \rangle,H^q(\mathcal{M} - M, \loc_{\mathbb{F}_2})) & \text{ if } p \text{ is odd}.
\end{cases}$$
Accordingly, there is a crude upper bound
\begin{eqnarray} \label{secondest}
\log |H^{*}(A'^{\bullet})| &=& \log | H^i_c( (\mathcal{M} - M) / \langle \sigma \rangle, \loc_{\mathbb{F}_2}) | \notag \\
&\leq& \sum_{p,q, \leq \dim \mathcal{M}} \log |E_2^{p,q}| \notag \\
&\leq& \sum_{p,q \leq \dim \mathcal{M}} \log | H^q(\mathcal{M} - M, \loc_{\mathbb{F}_2}) | \notag \\
&=& \dim \mathcal{M} \cdot \log | H^{*}(\mathcal{M} - M, \loc_{\mathbb{F}_2}) |  
\end{eqnarray}
Combining (\ref{firstest}), (\ref{secondest}) with the preliminary calculations of the previous section, we obtain (\ref{thirdest}).   
\end{proof}

\begin{rem}  
The restriction to involutions $\sigma$ was made for \emph{notational convenience only}.  A completely analogous argument can be carried out if $\sigma^p = 1$ and the exact same estimate  
\begin{eqnarray*}
\log NRT_{\sigma}(\loc) &=&  - \textcolor{blue}{\sum {}^\ast} \left(\log \left|H^i(\M, \loc)[p^{-1}]^{\sigma - 1} \right|  -  \frac{1}{p-1} \log \left|H^i(\M, \loc)[p^{-1}]^{P(\sigma)} \right| \right) \\
&+& O\left( \log|H^{*}(\mathcal{M}, \loc)[p^{\infty}]| +\log |H^{*}(\mathcal{M}, \loc_{\mathbb{F}_p})| + \log |H^{*}(M, \loc_{\mathbb{F}_p})| \right) \hspace{0.5cm}  (\ref{thirdest})_p
\end{eqnarray*}
is obtained.
\end{rem}

\subsection{Comparing $NRT_{\sigma}$ with $RT_{\sigma}$ on locally symmetric spaces}
\label{smalldifference}

Let $\loc \rightarrow \M$ be a unimodular metrized local system of free abelian groups \emph{over a locally symmetric space} $\M$ which is acted on equivariantly by an isometry $\sigma$ with $\sigma^p = 1.$ \smallskip

\begin{defn} 
If $B$ is a finite free abelian group of rank $n$ together with a Hermitian metric $h$ on $B_\C,$ we let the volume of $B$ denote the norm  $||e_1 \wedge ... \wedge e_n||_h$ for any basis $e_1,...,e_n$ of $M.$  The vector $e_1 \wedge ... \wedge e_n$ in $\wedge^n B_\C$ is independent of basis, up to sign, and so its norm is well-defined.
\end{defn}

\begin{prop} \label{rtvsnrt}
Suppose that $A^{\bullet}$ is one of the following complexes:
\begin{itemize}
\item
$A^{\bullet} = C^{\bullet}(\loc, \M ; K)$ is the group of $\loc$-valued cochains - with respect to a fixed triangulation $K$ of $\M$ extending a triangulation on $\M_{\sigma}$ - of a rationally acyclic, metrized, unimodular local system $\loc \rightarrow \M$ of free abelian groups acted on isometrically by $\langle \sigma \rangle \cong \mathbb{Z} / p\mathbb{Z}.$  

\item
$A^{\bullet} = \MS(X,\loc),$ the Morse-Smale complex for $\loc$ as above and some gradient vector field $X$ on $\M.$
\end{itemize}

Let $A'^{\bullet}$ be as defined in $\S \ref{commonnotation1}.$  Then
\begin{equation}
\log RT_{\sigma}(A^{\bullet}, h) = \log NRT_{\sigma}(A^{\bullet}) - \sum {}^{*} \log |H^i(A'^{\bullet})| + e \cdot \chi(\M_{\sigma})
\end{equation}
for some constant $e$ depending only on $\loc.$
\end{prop}

\begin{proof}
We prove the proposition for $A^{\bullet} =  C^{\bullet}(\loc, \M ; K),$ the proof for $A^{\bullet} = \MS(X, \loc)$ being identical.  We would like to evaluate
$$\log RT_{\sigma}(A^{\bullet}, h) = \frac{1}{2} \sum_{\pi \in \widehat{\mathbb{Z} / p\mathbb{Z}}} \log RT(A^{\bullet}[\pi], h|_{A^{\bullet}[\pi]}) \cdot \tr \; \pi(\sigma).$$
We have omitted the ``$g$" factor from the definition of $RT_{\sigma}(A^{\bullet}, h,g)$  because our complex is assumed rationally acyclic.  Because $A^{\bullet}_{\C}[P(\sigma)] = \oplus_{\chi \neq \mathbf{1}} A^{\bullet}_{\C}[\chi]$ - orthogonal direct sum - for every non-trivial character $\chi$ of $\mathbb{Z} / p\mathbb{Z},$ 
$$\log RT(A^{\bullet}[\chi], h|_{A^{\bullet}[\chi]}) = \frac{1}{p-1} \log RT(A^{\bullet}[P(\sigma)],  h|_{A^{\bullet}[P(\sigma)]}).$$
Thus,
\begin{eqnarray*}
\log RT_{\sigma}(A^{\bullet}, h) &=& \sum_{\pi \in \widehat{\mathbb{Z} / p\mathbb{Z}}} \log RT(A^{\bullet}[\pi], h|_{A^{\bullet}[\pi]}) \cdot \tr \; \pi(\sigma) \\
&=& \log RT(A^{\bullet}[\sigma - 1],  h|_{A^{\bullet}[\sigma - 1]}) + \frac{\zeta + ... + \zeta^{p-1}}{(p-1)} \log RT(A^{\bullet}[P(\sigma)], h|_{A^{\bullet}[P(\sigma)]}) \\
&=&  \log RT(A^{\bullet}[\sigma - 1],  h|_{A^{\bullet}[\sigma - 1]}) - \frac{1}{(p-1)} \log RT(A^{\bullet}[P(\sigma)], h|_{A^{\bullet}[P(\sigma)]}).
\end{eqnarray*}
We try to understand $A^{\bullet}[\sigma - 1], A^{\bullet}[P(\sigma)]$ by decomposing the action of $\sigma$ on $A^{\bullet}$ into orbits.  Let $\mathcal{C}$ be a set of representatives for the orbits of the $i$-cells under the action of $\sigma.$  For each $C \in \mathcal{C},$ we let $e^C_1,...,e^C_r$ be a basis for the global sections of $\loc$ restricted to $C$ (in particular $r = \mathrm{rank}(\loc)$ is fixed).   

\begin{itemize}
\item
Suppose first that $C$ is not fixed by the action of $\sigma.$  The $h$-volume of $\mathcal{O}_C[\sigma - 1],$ where $\mathcal{O}_C$ is the metrized abelian group $\mathcal{O}_C = \bigoplus \sigma^i\cdot C \otimes L(\sigma^i \cdot C),$ equals $\sqrt{p} \times || e^C_1 \wedge ... \wedge e^C_r ||_h = \sqrt{p},$ because the local system is unimodular and compatible with the metric $h,$ and $\sigma$ acts by isometries.  Similarly, it is readily seen that the $h$-volume of $\mathcal{O}_C[P(\sigma)]$ equals $\sqrt{p}^{p-1}.$        

\item
On the other hand, suppose that $C$ is fixed by the $\sigma$-action.  Then $C \subset \mathcal{M}_{\sigma}.$  Let $\mathcal{O}_C := C \otimes L(C).$  Furthermore, $c = \vol_h(C \otimes L(C)[\sigma - 1]), d = \vol_h(C \otimes L(C)[P(\sigma)])$ are independent of $C.$  Indeed, this follows because $L[\sigma - 1]$ and $L[P(\sigma)]$ are unimodular flat bundles \emph{over the fixed point set $\M_{\sigma}$}, and the volume form is induced by a particular invariant metric, which is a constant multiple of $h.$   
\end{itemize}

Because the $\mathcal{O}_C$ are mutually orthogonal for distinct representatives $C \in \mathcal{C},$ the above facts imply that the ratio of the norms $\vol_h(A^i[\sigma - 1]) / \alpha^i_{\sigma -1, \mathbb{Z}}$ and $\left\{ \vol_h(A^i[P(\sigma)]) / \alpha^i_{P(\sigma), \mathbb{Z}} \right\}^{\frac{1}{p-1}}$ (see $\S \ref{examples}$ for a discussion of $\alpha_{\mathbb{Z}}$) are almost exactly equal.  Their ratio is $(c/d^{\frac{1}{p-1}})^{\# i \text{-cells of } \M_{\sigma}}.$  Let $e = \frac{1}{p-1} \log d - \log c.$  Taking the alternating product over all $i,$ it follows that 
\begin{eqnarray*}
&{}& \log RT_{\sigma}(A^{\bullet}, h) (\sigma) \\
&=&  \log RT(A^{\bullet}[\sigma - 1],  h|_{A^{\bullet}[\sigma - 1]}) - \frac{1}{(p-1)} \log RT(A^{\bullet}[P(\sigma)], h|_{A^{\bullet}[P(\sigma)]}) \\
&=&  \log RT(A^{\bullet}[\sigma - 1],  \alpha_{\sigma - 1,\mathbb{Z}}) - \frac{1}{(p-1)} \log RT(A^{\bullet}[P(\sigma)], \alpha_{P(\sigma),\mathbb{Z}}) + e \cdot \chi(\M_{\sigma}) \\
&=&  \sum {}^{*} \left\{ \log |H^i(A^{\bullet}[\sigma - 1])| - \frac{1}{p-1} \log |H^i(A^{\bullet}[P(\sigma)])| \right\} + e \cdot \chi(\M_{\sigma})  \\
&=& \log NRT_{\sigma}(A^{\bullet}) - \sum {}^{*} \log |H^i(A'^{\bullet})| + e \cdot \chi(\M_{\sigma}).
\end{eqnarray*}
\end{proof}

\begin{cor} \label{twistedrtlss}
Let $\loc \rightarrow \M$ be a rationally acyclic, metrized, unimodular local system of free abelian groups acted on isometrically by $\langle \sigma \rangle \cong \mathbb{Z} / p\mathbb{Z}.$  Suppose that the fixed point set $\M_{\sigma}$ has Euler characteristic 0.  Then
\begin{eqnarray*}
\log RT_{\sigma}(\M,\loc) &=&  - \sum_i {}^{*} \left(\log \left|H^i(\M, \loc)[p^{-1}]^{\sigma - 1} \right|  -  \frac{1}{p-1} \log \left|H^i(\M, \loc)[p^{-1}]^{P(\sigma)} \right| \right) \\
&+& O\left( \log|H^{*}(\mathcal{M}, \loc)[p^{\infty}]| +\log |H^{*}(\mathcal{M}, \loc_{\mathbb{F}_p})| + \log |H^{*}(M, \loc_{\mathbb{F}_p})| \right)
\end{eqnarray*}
\end{cor}

\begin{proof}
This follows immediately by combining the analogue for $p$ of Proposition \ref{finitecomplexestimate} (see equation $(\ref{thirdest})_p$) with Proposition \ref{rtvsnrt}.
\end{proof}

\section[Equivariant Bismut-Zhang theorem]{Bismut-Zhang's equivariant Cheeger-M\"{u}ller theorem}
\label{dRecmt}

Bismut and Zhang in \cite[Theorem 0.2]{BZ} relate equivariant  Reidemeister torsion to equivariant analytic torsion, even for non-unitarily flat local systems. The difference between equivariant analytic torsion and equivariant Reidemeister torsion localizes to the fixed point set of the group action in a very controlled way, and we are able to compute this difference for unimodular local systems over certain locally symmetric spaces. \bigskip  

A model case for understanding both equivariant analytic torsion and equivariant Reidemeister torsion is that of products.   Let $L \rightarrow M$ be any metrized local system.  Then $L^{\boxtimes p} \rightarrow M^p$ is a metrized local system, equivariant with respect to the cyclic shift.  In $\S \ref{product},$ we compute both the equivariant analytic torsion and equivariant Reidemeister torsion of the product local system $L^{\boxtimes p} \rightarrow M^p.$ This case is of considerable importance because, after going through some contortions, the Bismut-Zhang formula (stated in Theorem \ref{BZ}) will allow us to directly relate the difference between analytic and Reidemeister torsion for general local systems to the corresponding difference for product local systems.  \bigskip

Let $\loc \rightarrow \M$ be an equivariant, metrized local system over Riemannian manifold $\M$ where $\langle \sigma \rangle$ of prime order $p$ acts compatibly on $\M$ and $\loc$ by isometries.  Let
\begin{equation} \label{error}
E(\M, \loc) =: \log \tau_{\sigma}(\loc) - \log RT_{\sigma}(\loc).
\end{equation}
In all situations that will concern us, the restriction of $\loc$ to the fixed point set $\M_{\sigma}$ is isomorphic, as a metrized, equivariant, local system, to $L^{\otimes p}$ for an appropriate local system over the fixed point set $\M_{\sigma}.$  Our general strategy for understanding $E(\M, \loc)$ is to prove that it equals $E(\M_{\sigma}^p, L^{\otimes p}).$  That such a comparison might be possible is suggested by the explicit form of the error term in the Bismut-Zhang theorem, which is local to a germ of the fixed point set $\M_{\sigma}.$

\begin{itemize}
\item
In $\S \ref{statementofBZ},$ we set up notation and state the version of the equivariant Cheeger-M\"{u}ller theorem proven by Bismut-Zhang.

\item
In $\S \ref{normalbundle},$ we relate the normal bundles of the inclusions $\M_{\sigma} \subset \M_{\sigma}^p$ and $\M_{\sigma} \subset \M,$ so as to allow a comparison of Morse functions on their respective normal neighborhoods.

\item
In $\S \ref{mariage},$ we compare Morse functions, connections, and local systems on $\loc \rightarrow \M$ and $L^{\boxtimes p} \rightarrow \M_{\sigma}^p.$  This allows us to conclude that $E(\M, \loc)$ equals $E(\M_{\sigma}^p, L^{\otimes p}).$ 
\end{itemize}

\subsection{Statement of the Bismut-Zhang Formula}   
\label{statementofBZ}

Notational setup:

\begin{itemize}
\item
$\loc \rightarrow \M$ denotes a metrized local system, with metric $h^\loc,$ with covariant derivative $\nabla^\loc$ for the canonical flat structure on $\loc.$  

\item
$\Gamma$ denotes a finite group acting on $\loc \rightarrow \M$ equivariantly by isometries.

\item
$f$ denotes a $\Gamma$-equivariant Morse function on $\M.$  Let $X$ be a $\Gamma$-invariant, weakly gradient-like vector field associated to $f,$ i.e. $X$ has non-degenerate critical points equal to the critical points of $f$ and $X(f) > 0$ away from the critical points of $f.$  Examples of these are provided by classical gradient vector fields $X = \mathrm{grad}_{g_0}(f)$ for $\Gamma$-invariant metrics $g_0.$

\item
$\sigma \in \Gamma$ acts on the normal bundle of $\M_{\sigma}$ in $\M$ by isometries and so induces an eigenbundle decomposition $N = \oplus N(\beta_j)$ where the eigenvalues of $\sigma$ acting on $N(\beta_j)$ are $e^{\pm i \beta_j}, \beta_j \in (0,\pi].$
\end{itemize}

Bismut and Zhang prove the following:

\begin{thm}[\cite{BZ}, Theorem 0.2] \label{BZ}
Let $f$ be a $\Gamma$-equivariant Morse function on $\M$ with associated weakly gradient-like vector field $X.$  Assume further that $X$ satisfies Morse-Smale transversality.  With notation as above, 
\begin{eqnarray} \label{BZformula}
&{}& 2 [ \log RT_{\Gamma}(\loc,f) - \log \tau_{\Gamma}(\loc) ](\sigma) \\ \notag 
&=& - \int_{\M_{\sigma}} \theta_{\sigma}(\loc, h^\loc) \wedge X^{*} \psi(T \M_{\sigma}, \nabla^{T \M_{\sigma}}) \\ \notag
&-& \frac{1}{4} \sum_{x \in \emph{Crit}(f) \cap \M_{\sigma}} (-1)^{\ind(f|_{\M_{\sigma}},x)} \sum_j \textcolor{blue}{(n_{+}(\beta_j,x) - n_{-}(\beta_j,x))} \cdot C_j \cdot \tr [\sigma | \loc_x]. 
\end{eqnarray}
In this formula, 

\begin{itemize}
\item
\textcolor{blue}{$(n_{+}(\beta_j,x) - n_{-}(\beta_j,x))$ denotes the number of positive minus the number of negative eigenvalues of the Hessian of $f$ acting on $N(\beta_j).$}

\item
$\textcolor{blue}{C_j = \Gamma' / \Gamma(\beta_j / 2\pi) + \Gamma' / \Gamma(1 - \beta_j/2\pi) - 2 \Gamma'(1).}$  These numbers are related to the equivariant torsion of odd dimensional spheres \cite[$\S 11$]{LR}.

\item
$\psi$ is a current on $T \M_{\sigma}$ whose restriction to $T \M_{\sigma} - 0$ transgresses the Euler class for $T \M_{\sigma}.$  This is the very same current used by Bismut-Cheeger in \cite{BC}.

\item
$\theta_{\sigma}(\loc, h^\loc) = \tr(\sigma \cdot \omega(\loc, h^\loc))$ where $\omega(\loc,h^\loc) = (h_\loc)^{-1} \circ \nabla^{\Hom(\loc, \check{\loc}) } (h^L).$  \\
The connection $\nabla^\loc$ induces a connection $\nabla^{\Hom(\loc, \check{\loc})}$ on $\Hom(\loc, \check{\loc}),$ and we view the metric tensor $h^\loc$ as a global section of $\Hom(\loc, \check{\loc}).$  Then $\nabla^\loc(h^\loc)$ can be viewed as a $\Hom(\loc, \check{\loc})$-valued 1-form, which after composing with $(h^\loc)^{-1}$ becomes an $\Hom(\loc,\loc)$-valued 1-form. 
\end{itemize}
\end{thm}

For our later purposes, it will be crucial for us to know that the right side of equation (\ref{BZformula}) collapses significantly because the closed $1$-form $\theta_{\sigma}(\loc, h^\loc)$ appearing in the above integral over $\M_{\sigma}$ often vanishes.

\begin{lem} \label{unimodular}
\textcolor{blue}{Let $\loc|_{\M^\sigma} = \bigoplus_{\epsilon} \loc_{\epsilon}$ be the (orthogonal) eigenbundle decomposition of $\sigma$ acting on $\loc|_{\M^{\sigma}}.$ Suppose that all of the local systems $\det(\loc_\epsilon) \rightarrow \M^\sigma$ are unitarily flat.}  Then 
$$\theta_{\sigma}(\loc, h^\loc) = 0.$$
In particular, $\theta_{\sigma}(\loc, h^\loc) = 0$ whenever \textcolor{blue}{every $\det(\loc_\epsilon)$} is the trivial local system. 
\end{lem}

\begin{proof}
Recall that
$$\theta_{\sigma}(\loc, h^\loc) = \tr \{ \sigma \cdot \omega(\loc, h^\loc) \} \text{ where } \omega(\loc, h^\loc) = (h^\loc)^{-1} \nabla^\loc(h^\loc).$$
\textcolor{blue}{Because the decomposition $\loc = \bigoplus_\epsilon \loc_\epsilon$ is orthogonal, 
$$\theta_{\sigma}(\loc, h^\loc) = \sum_\epsilon \epsilon \; \tr(\omega(\loc_\epsilon, h^{\loc_\epsilon})).$$}
Because trace is the infinitesimal determinant, we readily check that for every $\epsilon,$
\textcolor{blue}{$$\tr \{ \omega(\loc_\epsilon, h^{\loc_\epsilon}) \} = \omega(\det(\loc_\epsilon), h^{\det(\loc_\epsilon)}) \text{ where } \omega(\det(\loc_\epsilon), h^{\det(\loc_\epsilon)}) := (h^{\det(\loc_\epsilon)})^{-1} \nabla^{\det(\loc_\epsilon)}(h^{\det(\loc_\epsilon)}).$$}
But the form \textcolor{blue}{$\omega(\det(\loc_\epsilon), h^{\det(\loc_\epsilon)})$} measures the obstruction to the bundle \textcolor{blue}{$\det(\loc_\epsilon) \rightarrow \M^\sigma$} being unitarily flat relative to the pair \textcolor{blue}{$\nabla^{\det(\loc_\epsilon)}, h^{\det(\loc_\epsilon)}.$}  Therefore, if \textcolor{blue}{$\det(\loc_\epsilon) \rightarrow \M^\sigma$} is unitarily flat, \textcolor{blue}{$\omega(\det(\loc_\epsilon), h^{\det(\loc_\epsilon)}) = 0$} and the result follows.
\end{proof}

\subsection{Intrinsic identification of the normal bundle}
\label{normalbundle}

We aim to compare the normal bundles of the inclusions $\M_{\sigma} \xrightarrow{i_1} \M$ and $\M_{\sigma} \xrightarrow{i_2} \M_{\sigma}^p,$ with a view to comparing error terms of two different applications Bismut-Zhang theorem $\ref{BZ},$ one on a local system over $\M$ and another on a local system over $\M_{\sigma}^p.$  \bigskip

Let $\M$ denote a Galois stable locally symmetric space associated to the group $R_{E/F} \G$ for a cyclic degree $p$ Galois extension $E/F.$ 

\begin{prop} \label{isometric}
The normal bundles of the inclusions
$$i_1^{*} N(\M_{\sigma} \subset \M_{\sigma}^p) \rightarrow \M_{\sigma} \text{ and } i_2^{*} N(\M_{\sigma} \subset \M) \rightarrow \M_{\sigma}$$
are isometric. 
\end{prop}

\begin{proof}
For ease of notation, we will assume that $F$ is imaginary quadratic (see Remark \ref{ssnormalbundle}).

Fix a complex embedding $\iota$ of $F.$  For the cyclic degree $p$-extension $E/F,$ we have the symmetric space

$$S = \prod_{v|\iota} \mathbb{H}^3_v,$$

which is the universal cover of $\M.$  Correspondingly, there is a decomposition of the tangent bundle

$$TS = \bigoplus_{v | \iota} T \mathbb{H}^3_v.$$

This decomposition is Galois invariant and the subbundles $T \mathbb{H}_v$ are individually invariant under $\G(E_{\mathbb{R}}).$  Fix a Galois-stable path component $M^0$ of either $M_{\sigma} \times ... \times M_{\sigma}$ or $\M.$  Let $M^0_{\sigma}$ denote the union of those path components of $M_{\sigma}$ contained in $M^0.$ \bigskip  

There is a Galois-equivariant covering $S \rightarrow M^0.$  The above decomposition of $TS$ descends to a Galois-equivariant decomposition of the tangent bundle of $M^0$:
$$TM^0 = V_1 \oplus ... \oplus V_p.$$ 
The restriction $V$ of the tangent bundle $TM^0$ to $M^0_{\sigma}$ carries an action of $\Gamma_{E/F} = \langle \sigma \rangle.$  The compositions
$$\phi_i := V^{\sigma} \hookrightarrow V \xrightarrow{\pi_i} V_i $$
are all $p^{-1/2}$ times an isometry.  Indeed, any invariant vector must have ${\sigma}^i \cdot 1$-component $\sigma^i \cdot v$ for some vector $v.$  This vector has norm $p^{1/2} || v||$ and its projection has norm $|| v||.$  Thus, 
$$\phi = \oplus (p^{1/2} \phi_i)^{-1}: V = \oplus V_i|_{M^0_{\sigma}} \rightarrow (V^{\sigma})^{\oplus p}$$
is an isometry which identifies the normal bundle of $M^0_{\sigma} \subset M^0$ with the orthogonal complement $\Delta(V^{\sigma})^{\perp}$ of the diagonal in $(V^{\sigma})^{\oplus p}.$  This intrinsic description of the normal bundle shows that the normal bundle of the inclusions $i_1: \M_{\sigma} \subset \M_{\sigma}^p$ and $i_2: \M_{\sigma} \subset \M$ are $\Gamma$-equivariantly isometrically isomorphic.
\end{proof}

\begin{rem} \label{ssnormalbundle}
Every symmetric space $S$ admits a decomposition as a product $G_1 / K_1 \times ... \times G_n/K_n.$  The decomposition $\g_1 / \comp_1 \times ... \times \g_n / \comp_n,$ isomorphic to the tangent space at $(eK_1,...,eK_n),$ as a $K_1 \times ... \times K_n$ representation is multiplicity-free.  Thus, it admits a canonical decomposition as a sum of irreducibles and so induces a canonical decomposition of the tangent bundle $TS.$  If a finite order automorphism $\sigma$ of the group $G_1 \times ... \times G_n$ preserves $K_1 \times ... \times K_n$ and normalizes a discrete group $\Gamma$ of isometries of $S,$ then this decomposition descends to give a $\sigma$-stable decomposition of $T(\Gamma \backslash S).$  The above identification of normal bundles carries through in this generality.     
\end{rem}

\subsection{Le mariage de la carpe et du lapin}
\label{mariage}

The error term in the Bismut-Zhang theorem sees very little of the space $\M,$ only a $\Gamma$-invariant Morse function $f$ on an arbitrarily small neighborhood of the fixed point set $\M_{\sigma}.$  The previous section allows us to relate the Morse function $f$ in a neighborhood of $\M_{\sigma} \subset \M$ to a different $\Gamma$-invariant Morse function $f'$ on $\M_{\sigma} \subset (\M_{\sigma})^p.$  The goal of this section will be to relate all parts of the error terms in two different applications of the Bismut-Zhang formula, the first on $\M$ and the second on $\M_{\sigma}^p$.  In the notation of $(\ref{error})$ from the introduction to $\S \ref{dRecmt},$ we will prove that 
\begin{equation} \label{producterror}
E(\M, \loc) = E(\M_{\sigma}^p, L^{\boxtimes p})
\end{equation}
for an appropriate local system $L \rightarrow \M_{\sigma}$ (to be stated more precisely in Proposition $\ref{finalmariage}$).  This represents progress because, as we will see in $\S \ref{product},$ both the twisted analytic torsion and the twisted Reidemeister torsion in the case of the product $(M_{\sigma})^p$ can be explicitly computed, thus enabling us to understand the left side of \eqref{producterror}.

\subsubsection{Transport of structure}
\label{morsecomparison}

In this section, we will provide all of the necessary ingredients for comparing the error terms in two applications of the Bismut-Zhang formula; in $\S \ref{finalmariage},$ this will enable us to prove that
$$\log RT_{\Gamma}(\M, \loc, X_1) - \log \tau_{\Gamma}(\M, \loc) = \log RT_{\Gamma}(\M_{\sigma}^p, L^{\boxtimes p}, X_2) - \log \tau_{\Gamma}(\M_{\sigma}^p, L^{\boxtimes p})$$
for some special choices of gradient vector fields $X_1, X_2.$  \bigskip

Let $N_1 \rightarrow M_{\sigma} \leftarrow N_2$ denote the normal bundles of $i_1: M_{\sigma} \subset M, i_2: M_{\sigma} \subset (M_{\sigma})^p$ respectively.  The previous section produces an explicit isometric isomorphism $N_1 \xrightarrow{\Phi} N_2.$

\subsubsection*{Comparison of Morse functions}

Let $f_1$ be a $\Gamma$-invariant Morse function on a small exponential neighborhood of $M_{\sigma} \subset M$ of radius $r.$  Because $\sigma$ acts by isometries on $M,$ this neighborhood is $\Gamma$-invariant.  We construct a matching function $f_2$ on the exponential neighborhood of $M_{\sigma} \subset (M_{\sigma})^p$ of radius $r$ by
$$f_2(\exp_p(\Phi(Y))) := f_1(\exp_p(Y))$$
for any $Y \in B_p(r) \subset (N_1)_p.$  $f_2$ is a $\Gamma$-invariant Morse function on this exponential neighborhood.  Furthermore, because
$$\exp_p(Y) \xrightarrow{\Phi} \exp_p(\phi(Y))$$ 
is a $\Gamma$-equivariant diffeomorphism, all of the critical points and indices of the two Morse functions are equal.  For example,

$$\ind(f_1,N_1(\beta_j); x) = \ind(f_2, N_2(\beta_j); x)$$

because both can be computed by exponentiating $N(\beta_j)(x)$ to form a submanifold $\Pi$ a neighborhood of $x$ and computing the Morse index of $f|_{\Pi}.$ But clearly, the indices of $f_1|_{\Pi}$ and $f_2|_{\Phi(\Pi)}$ are equal, as they are related by a diffeomorphism.\bigskip

\textbf{Remark}.  It is easily seen that $f_2$ can be extended to a $\Gamma$-invariant Morse function on $M_{\sigma} \times ... \times M_{\sigma}.$  Indeed, a generic function on the quotient $(M_{\sigma}^p - \Delta(M_{\sigma})) / \langle \sigma \rangle$ is Morse.  Extend $f_2,$ which descends to a Morse function on this quotient, by a generic bump function.

\subsubsection*{Comparison of local systems}

All metrized local systems on $\loc \rightarrow \M$ that we'll encounter in our applications will have the property that $\loc|_{\M_{\sigma}} = L^{\otimes p}$ for a local system  $L \rightarrow \M_{\sigma},$ endowed with the obvious tensor product metric.  Thus, for each $x \in \M_{\sigma},$

$$\tr(\sigma | \loc_x) = \tr(\sigma | L^{\boxtimes p}_x).$$

%
%
%
%
%
%

\subsubsection{Concluding the comparison}

Combining all of the comparisons of $\S \ref{morsecomparison},$ we can prove that the error term in the Bismut-Zhang formula arising from the equivariant unimodular metrized local system $\loc \rightarrow \M$ exactly equals the error term arising from the local system $L^{\boxtimes^p} \rightarrow \M_{\sigma}^p.$  

\begin{prop} \label{finalmariage}
Let $\loc \rightarrow \M$ and $L \rightarrow \M_{\sigma}$ be matching unimodular local systems and let $f_1$ and $f_2$ be invariant Morse functions which match in the sense described above.  Then there are invariant gradient vector fields $X_1$ associated to $f_1$ and $X_2$ associated to $f_2$ satisfying Morse-Smale transversality for which 
\begin{equation} \label{finalmariageequation}
[\log RT_{\Gamma}(\loc \rightarrow \M, X_1) - \log \tau_{\Gamma}(\loc \rightarrow \M) ](\sigma) = [\log RT_{\Gamma}(L^{\boxtimes p} \rightarrow \M_{\sigma}^p,f_2) - \log \tau_{\Gamma}(L^{\boxtimes p} \rightarrow \M_{\sigma}^p) ](\sigma).
\end{equation}
\end{prop}

\begin{proof}
For any invariant Morse function $f_1$ defined on an exponential neighborhood $N_1$ of $\M_{\sigma} \subset \M,$ we have constructed an invariant Morse function $f_2$ which matches it on a $\Gamma$-equivariantly diffeomorphic tubular neighborhood $N_2$ of $\M_{\sigma} \subset \M_{\sigma}^p.$  For these matching Morse functions, the discrete parts of the error term of the Bismut-Zhang formula (\ref{BZformula}) match: 
$$\sum_{x \in \mathrm{Crit}(f_1) \cap \M_{\sigma}} (-1)^{\ind(f_1|_{\M_{\sigma}},x)} \sum_j \ind(f,  N(\beta_j); x) \cdot C_j \cdot \tr [\sigma | \loc_x]$$
\begin{equation} \label{discrete}
=  \sum_{x \in \mathrm{Crit}(f_2) \cap \M_{\sigma}} (-1)^{\ind(f_2|_{\M_{\sigma}},x)} \sum_j \ind(f,  N(\beta_j); x) \cdot C_j \cdot \tr [\sigma | L^{\boxtimes p}_x].
\end{equation}
Furthermore, because $\loc$ is unimodular by assumption, Lemma \ref{unimodular} implies that the continuous part of the error term of the Bismut-Zhang formula is zero, no matter which choices of gradient vector fields $X_1, X_2$ we ultimately make.  However, in order for this error term to compute the correct quantity, we need to 

\begin{itemize}
\item[(a)]
extend $f_1$ from $N_1$ to an invariant Morse function on $\M$ and $f_2$ from $N_2$ to an invariant Morse function on $\M_{\sigma}^p.$

\item[(b)]
find $\Gamma$-invariant, weakly gradient-like vector fields $X_1, X_2$ associated to $f_1, f_2$ which satisfy Morse-Smale transversality.
\end{itemize}

The first item (a) can be readily accomplished.  Indeed, $\langle \sigma \rangle$ acts freely on both $\M - \M_{\sigma}$ and $\M_{\sigma}^p - \M_{\sigma},$ and so $(\M - \M_{\sigma}) / \langle \sigma \rangle$ and $(\M_{\sigma}^p - \M_{\sigma}) / \langle \sigma \rangle$ are manifolds.  The Morse functions $f_1$ on $(N_1 -  \M_{\sigma}) / \langle \sigma \rangle$ and $f_2$ on $(N_2 - \M_{\sigma}) / \langle \sigma \rangle$ can be extended randomly to Morse functions on $(\M - \M_{\sigma}) / \langle \sigma \rangle$ and $(\M_{\sigma}^p - \M_{\sigma}) / \langle \sigma \rangle$ by genericity of Morse functions.  Their pullbacks can be glued with the original $f_1$ and $f_2$ to give invariant Morse functions, which we continue to call $f_1, f_2.$ \medskip

The second item (b) is more delicate.  Fix a metric $g_1$ on $\M,$ such as the group invariant metric.  We make a somewhat special choice of $f_1;$  we require that the Hessian $d^2f_1(x)$ is negative definite on $N_x,$ where $N$ denotes the normal bundle to $\M_{\sigma} \subset \M$ at all critical points on $M_{\sigma};$ this can be accomplished by letting $f_1,$ in exponential coordinates, equal $f_1(x,v) = f(x) + g_0(v)$ for any Morse function $f$ on $\M_{\sigma}$ and any metric $g_0$ on the normal bundle. \smallskip

By \cite[Theorem 1.8]{BZ}, for any such choice of $f_1$, there is a metric $\tilde{g}_1$ equal to $g_1$ on a neighborhood of all critical points of $f_1$ on $\M_{\sigma}$ for which 
$$X_1 := \mathrm{grad}_{\tilde{g}_1}(f_1) \text{ satisfies Morse-Smale transversality}.$$  
Because $\tilde{g}_1$ equals $g_1$ in a neighborhood of the critical points in $\M_{\sigma},$ the equality of the error terms from equation (\ref{discrete}) is preserved.  

Let $f_2$ be the Morse function matching $f_1.$  We can play the same game by modifying a fixed metric $g_2$ on $\M_{\sigma}^p,$ such as the group invariant metric, to a new metric $\tilde{g}_2$ for which 
$$X_2 := \mathrm{grad}_{\tilde{g}_2}(f_2) \text{ satisfies Morse-Smale transversality}.$$
Finally, the assumptions of the Bismut-Zhang theorem are satisfied, and we can conclude that
\begin{equation} \label{badmariage}
\begin{split}
&{} [\log RT_{\Gamma}(\loc \rightarrow \M, X_1) - \log \tau_{\Gamma}(\loc \rightarrow \M, \tilde{g}_1) ](\sigma) \\
&= [\log RT_{\Gamma}(L^{\boxtimes p} \rightarrow \M_{\sigma}^p,f_2) - \log \tau_{\Gamma}(L^{\boxtimes p} \rightarrow \M_{\sigma}^p, \tilde{g}_2) ](\sigma).
\end{split}
\end{equation}
Note that the analytic torsion terms $\log \tau_{\sigma}(\loc \rightarrow \M, \tilde{g}_1)$ and $ \log \tau_{\sigma}(L^{\boxtimes p} \rightarrow \M_{\sigma}^p, \tilde{g}_2)$ appear to depend on the metrics $\tilde{g}_1$ and $\tilde{g}_2$ over which we have no control.  Such a dependence would be devastating because, ultimately, we can only calculate information concerning comparisons or growth of twisted analytic torsion relative to \emph{group invariant metrics}.  That being said, we are rescued by the anomaly formula of \cite[Theorem 0.1]{BZ2}, which states that
$$\log \tau_{\sigma}(\loc \rightarrow \M, g_1) - \log \tau_{\sigma}(\loc \rightarrow \M, \tilde{g}_1)$$
\begin{equation}
= \int_{\M_{\sigma}} A \wedge e(T \M_{\sigma}, \nabla_{\tilde{g_1}}^{T \M_{\sigma}}) - \int_{\M_{\sigma}} \theta_{\sigma}(\loc, h^\loc) \wedge B
\end{equation}
where $A, B$ are some differerential forms that we will not specify, $e(T \M_{\sigma}, \nabla_{\tilde{g}_1}^{T \M_{\sigma}})$ is the Euler form of $T \M_{\sigma}$ relative to the Levi-Civita connection of $T\M_{\sigma}$ associated to the metric $\tilde{g}_1,$ and $\theta_{\sigma}(\loc, h^\loc)$ is the closed 1-form from Lemma \ref{unimodular}.  According to Lemma \ref{unimodular}, the form $\theta_{\sigma}(\loc, h^\loc)$ vanishes identically because $\loc \rightarrow \M$ is unimodular.  Also, because $\M_{\sigma}$ is an odd dimensional manifold in our case, the Euler form is identically zero.  It follows that
$$\log \tau_{\sigma}(\loc \rightarrow \M, g_1) - \log \tau_{\sigma}(\loc \rightarrow \M, \tilde{g}_1) = 0.$$
The proposition now follows.
\end{proof}

\section{Calculations on a product}
\label{product}

%

Let $K$ be a triangulation of a Riemannian manifold $M,$ and let $L \rightarrow M$ be a metrized local system of free abelian groups.  Let $K'$ be an equivariant refinement of $K^p$ which extends a triangulation of the diagonal $M$ inside the product $M^p.$  In this section, we aim to prove that the difference 
$$[\log RT_{\Gamma}(L^{\boxtimes p} \rightarrow M^p,K') - \log \tau_{\Gamma}(L^{\boxtimes p} \rightarrow M^p) ](\sigma).$$
is small, where the group $\Gamma = \langle \sigma \rangle$ acts by cyclic permutation.  We will separately relate the twisted analytic and twisted Reidemeister torsion appearing in the above equation to their untwisted counterparts $\log RT(L \rightarrow M)$ and $\log \tau(L \rightarrow M)$ respectively.

\begin{itemize}
\item
In $\S \ref{twistedAprod},$ we prove that
$$\log \tau_{\sigma}(M^p, L^{\boxtimes p}) = p \log \tau(M, L),$$ 
where $\sigma$ denotes the cyclic shift.

\item
In $\S \ref{twistedRprod},$ we prove that
$$\log NRT_{\sigma}(M^p, L^{\boxtimes p}) = p \log RT(M, L).$$
We prove this using the close relationship between $NRT_{\sigma}$ and analytic torsion (see \eqref{acyclictwistedguess} and Definition \ref{nrt}).  

\item
In $\S \ref{smalldifferenceproduct},$ we prove that the difference
$$\log RT_{\sigma}(M^p, L^{\boxtimes p}) - \log \tau_{\sigma}(M^p, L^{\boxtimes p})$$
is often zero.  This is a direct consequence of L\"{u}ck's variant of the equivariant Cheeger-M\"{u}ller theorem when $L \rightarrow M$ is \emph{unitarily flat}, but we prove that this difference is zero for arbitrary local systems $L$ over odd dimensional locally symmetric spaces $M$ for which $L_{\mathbb{F}_p}$ is self-dual.  
\end{itemize}

\subsection{Twisted analytic torsion of a product}
\label{twistedAprod}

Let $M$ be a compact Riemannian manifold together with a metrized local system $L \rightarrow M.$  We next compute the equivariant torsion of $L^{\boxtimes n} \rightarrow M^n$ with respect to a cyclic shift $\sigma.$  

\begin{prop} \label{torsionproduct}
If $L^{\boxtimes n} \rightarrow M^n$ is equipped with its product metric, then
\begin{equation} \label{torsionproductequation}
\log \tau_{\sigma}(L^{\boxtimes n}) = n \left[  \log \tau(L) - \log(n)Z_L(0)  \right].
\end{equation}
\end{prop}

\begin{proof}
Let $\pi_i: M^n \rightarrow M$ denote the $i$th coordinate projection. Let $\Delta_j$ be the $j$-form Laplacian acting on $\Omega^j(M^n, L^{\boxtimes n}).$  We let $E_{\lambda, j}$ denote the $\lambda$-eigenspace of the $j$-form Laplacian of $M.$ Note that the $\lambda$-eigenspace of $\Delta_{j, M^n}$ is spanned by the image of 
\begin{eqnarray*}
\bigoplus  E_{\lambda_1,i_1} \otimes ... \otimes E_{\lambda_n,i_n} &\rightarrow& \Omega_j(M^n, L^{\boxtimes n}) \\ 
\omega_1 \otimes ... \otimes \omega_n &\mapsto& \pi_1^{*} \omega_1 \wedge ... \wedge \pi_n^{*} \omega_n;
\end{eqnarray*}
the sum ranges over all $\lambda_1 + ... + \lambda_n = \lambda$ and $i_1 + ... + i_n = j.$  

The isometry $\sigma$ acts as a signed permutation on the vectors $\pi_1^{*} \omega_1 \wedge ... \wedge \pi_n^{*} \omega_n,$ where $\omega_k$ runs over a basis of $E_{\lambda_k, i_k}.$  Thus, such a basis vector contributes to the trace exactly when $\omega_1 = ... = \omega_n.$  So in particular, $j$ must be divisible by $n.$  Say $j = an$ and $\lambda = n \lambda'.$  Then the trace of $\sigma$ acting on the image of $E_{\lambda',a} \otimes ... \otimes E_{\lambda',a}$ is readily seen to be
$$(-1)^{a^2(n-1)} \dim E_{\lambda',a}.$$
It follows that
\begin{eqnarray*}
\sum_j (-1)^j j \zeta_{j,\sigma}(s) &=&  \sum_a (-1)^{na}na  (-1)^{a^2(n-1)} \zeta_a(s) n^{-s}\\
&=& n \sum_a (-1)^{n(a^2 - a) - a^2} a \zeta_a(s) n^{-s} \\
&=& n \sum_a (-1)^a a \zeta_a(s) n^{-s}.
\end{eqnarray*}

Differentiating at $s = 0$ gives the result.
\end{proof}

\begin{rem} 
The $\log(n) Z_L(0)$ summand in equation (\ref{torsionproductequation}) arises because the diagonal $\Delta(M^n) \subset M^n$ has two possible metrics: the induced metric from the product metric on $M^n$ and the given metric on $M.$  The above calculations were made with respect to the given metric on $M.$  If we had instead chosen the metric induced on the diagonal, the second summand in equation (\ref{torsionproductequation}) would not appear.
\end{rem}

\subsection{Twisted Reidemeister torsion on a product}
\label{twistedRprod}

Let $C$ be a finite, $\mathbb{Q}$-acyclic complex acted on by $\mathbb{Z} / p\mathbb{Z} = \langle \sigma \rangle.$  For convenience, we recall the definition of naive Reidemeister torsion from Definition $\ref{nrt}.$  
\begin{equation}
\log NRT_{\sigma}(C^{\bullet}) := \left\{ \sum {}^{*} \log |H^i(C^{\sigma - 1})| - \frac{1}{p-1} \log |H^i(C^{P(\sigma)})| \right\} + \left\{ \sum {}^{*} \log | H^i(C') | \right\}. 
\end{equation}  
where $C' := C / \left( C[\sigma - 1] \oplus C[P(\sigma)] \right).$

\begin{lem} \label{twistedproductnrt}
Suppose that $C = A^{\otimes p},$ with $\sigma$ acting by cyclic permutation.  Then
$$\log NRT_{\sigma}(A^{\otimes p}) = p \log RT(A).$$
\end{lem}

\begin{proof}
By Lemma \ref{twistedfinitecomplex}, we have that 
$$\log NRT_{\sigma}(A^{\otimes p}) = \log \tau_{\sigma}(A^{\otimes p}).$$
By a computation identical to that of $\S \ref{torsionproduct},$ we see that
$$\log \tau_{\sigma}(A^{\otimes p}) = p \log \tau(A) = p \log RT(A),$$ 
and we are done.
\end{proof}

\subsection{Proof that $\log RT_{\sigma}(L^{\boxtimes p} \rightarrow M^p,X) - \log \tau_{\sigma}(L^{\boxtimes p} \rightarrow M^p)$ is often 0}
\label{smalldifferenceproduct}

\begin{thm} \label{productcheegermuller}
Suppose that $M$ is an odd dimensional, compact manifold equipped with a triangulation $K_0.$  Let $L \rightarrow M$ be a metrized local system of free abelian-groups with $L_{\mathbb{F}_p}$ \emph{self-dual}.  Let $X$ be a vector field on $M^p$ which is weakly gradient-like for a Morse function $f$ on $M^p$ and which satisfies Morse-Smale transversality.  Then 
$$\log RT_{\sigma}(L^{\boxtimes p} \rightarrow M^p,X) - \log \tau_{\sigma}(L^{\boxtimes p} \rightarrow M^p) = 0$$
\end{thm}

\begin{proof}
For shorthand, let $C = (A^{\bullet})^{\otimes p},$ where $A =  C^{\bullet}(M,L;K_0)$; then $ C^{\bullet}(M^p, L^{\boxtimes p}; K_0^p) \cong A^{\otimes p}.$  We also let $D = \MS^{\bullet}(X, L^{\boxtimes p}).$  By Proposition \ref{rtvsnrt}, there is an equality 
\begin{eqnarray*}
&{}& \log RT_{\sigma}(L^{\boxtimes p} \rightarrow M^p,K) - \log \tau_{\sigma}(L^{\boxtimes p} \rightarrow M^p) \\
&=& \log NRT_{\sigma}(D) -  \log \tau_{\sigma}(L^{\boxtimes p} \rightarrow M^p) - \sum {}^{*} \log |H^i( D' )| + e \cdot \chi(M). 
\end{eqnarray*}
By Lemma \ref{quasiisomorphism}, $NRT_{\sigma}$ is invariant under equivariant chain homotopy (see also \cite[$\S 5,$ Proposition 8]{LR}).  However, by the main result of \cite{Ill}, the triangulation of $M^p$ by unstable manifolds of $X$ and the product triangulation $K_0^p$ - both smooth triangulations equivariant for the cyclic shift on $M^p$ - admit a common smooth equivariant refinement.  Therefore, the complexes $C$ and $D$ are equivariantly chain homotopic, from which it follows that $NRT_{\sigma}(D) = NRT_{\sigma}(C).$   The latter expression thus equals 
\begin{eqnarray*}
&=& \log NRT_{\sigma}(C) -  \log \tau_{\sigma}(L^{\boxtimes p} \rightarrow M^p) - \sum {}^{*} \log |H^i( D' )| + e \cdot \chi(M) \\
&=& p \log RT(A^{\bullet}) - p \log \tau(L \rightarrow M) - \sum {}^{*} \log |H^i( D' )|  + 0\\
&=& 0 - \sum {}^{*} \log |H^i( D' )|  \\
&=& - \sum {}^{*} \log |H^i( D' )|. 
\end{eqnarray*}
The transition from the first line to the second follows by Lemma \ref{twistedproductnrt} and because the odd dimensional compact manifold $M$ has $\chi(M) = 0$; the transition from the second to third line follows by the untwisted Cheeger-M\"{u}ller theorem.  But as observed in Proposition \ref{geominterpretation},
$$H^i( D' ) = H_c^i((M^p - M) / \langle \sigma \rangle,L^{\boxtimes p}_{\mathbb{F}_p}).$$
We thus continue
\begin{eqnarray*}
&=& - \sum {}^{*} \log |H^i( D' )| \\
&=& - \log p \cdot \chi_c((M^p - M) / \langle \sigma \rangle,L^{\boxtimes p}_{\mathbb{F}_p}) \\
&=& - \log p \cdot p \cdot  \chi_c(M^p - M, L^{\boxtimes p}_{\mathbb{F}_p}) \\ 
&=& -\log p \cdot p \cdot [ \chi(M^p, L^{\boxtimes p}_{\mathbb{F}_p}) - \chi(M, L^{\boxtimes p}_{\mathbb{F}_p}|_M) ] \\
&=& - \log p \cdot p \cdot [\chi(M, L_{\mathbb{F}_p} )^p - \chi(M, L^{\otimes p}_{\mathbb{F}_p})].
\end{eqnarray*}

Since $M$ is an odd dimensional manifold and $L_{\mathbb{F}_p}$ is a self-dual local system on $M,$ Poincar\'{e} duality implies that both Euler characteristics appearing in the final line above vanish.  We conclude that
$$ \log RT_{\sigma}(L^{\boxtimes p} \rightarrow M^p,X) - \log \tau_{\sigma}(L^{\boxtimes p} \rightarrow M^p) = 0$$
on the nose.
\end{proof}

\begin{cor} \label{twistedatequalstwistedrt}
Let $\loc \rightarrow \M$ be an equivariant, metrized, rationally acyclic local system of free abelian groups over a locally symmetric space $\M$ acted on equivariantly and isometrically by $\langle \sigma \rangle \cong \mathbb{Z} / p\mathbb{Z}.$  Suppose further that the restriction to the fixed point set $\loc|_{\M_{\sigma}} = L^{\otimes p}$ for $L \rightarrow \M_{\sigma}$ self-dual and that $\M_{\sigma}$ is odd dimensional.  It follows that 
$$\log \tau_{\sigma}(\M, \loc) = \log RT_{\sigma}(\M,\loc).$$ 
\end{cor}

\begin{proof}
This follows immediately by combining Theorem \ref{productcheegermuller} with Proposition \ref{finalmariage}.
\end{proof}

\begin{exam} \label{productquaternionalgebras}
Let $\G$ be the adjoint group of the units of a quaternion algebra over a number field $F.$  Let $K \subset \G(F_\R)$ be a maximal compact subgroup.  Let $\rho: \G \rightarrow \GL(V)$ be an algebraic representation defined over $F.$  Fix a lattice $\mathcal{O} \subset V;$ let $U_0 \subset \G(\finadele_F)$ be its stabilizer, a compact open subgroup.  Let $U \subset \finadele_F$ be any compact open subgroup contained in $U_0.$  To the representation $\rho$ is associated a local system of $O_F$-modules $L_{\rho} \rightarrow M_U = \G(F) \backslash \G(\adele_F) / KU$ 
(see \cite[\S 2]{Lip2}). \medskip

In this case, every $L_{\rho, \mathbb{F}_p}$ is self-dual.  Furthermore, $M_U$ is odd dimensional precisely when the number of complex places of $F$ is odd.  In such cases, $L_{\rho} \rightarrow M_U$ satisfies the requirements of Theorem \ref{productcheegermuller} from which it follows that

$$\log RT_{\sigma}(L^{\boxtimes p} \rightarrow M^p, X) = \log \tau_{\sigma}(L^{\boxtimes p} \rightarrow M^p).$$
\end{exam}

\appendix

\section[Naive R-torsion of a tensor product]{Naive Reidemeister torsion of a tensor product complex}
\label{nrtappendix}

\subsection{Robustness properties of $NRT_{\sigma}$}

The goal of this section is to show that $NRT_{\sigma},$ as defined in Definition \ref{nrt}, is very well-behaved in two respects:   

\begin{itemize}
\item[(a)]
If $C \rightarrow E$ is a chain homotopy of $\mathbb{Z}[ \sigma ]$-complexes of free $\mathbb{Z}$-modules, then $NRT_{\sigma}(C) = NRT_{\sigma}(E).$

\item[(b)]
The naive Reidemeister torsion is additive for tensor complexes, i.e. if $C_0 = A^{\otimes p}, C_1 = B^{\otimes p}$ for complexes $A,B$ of finite free $\mathbb{Z}$-modules with $\sigma$ acting on $C_0,C_1$ by cyclic permutation, then 
$$\log NRT_{\sigma}(C_0 \oplus C_1) = \log NRT_{\sigma}(C_0) + \log NRT_{\sigma}(C_1).$$ 
\end{itemize}

We now outline a proof of these two properties.

\begin{lem}[(a) chain homotopy invariance] \label{quasiisomorphism}
$NRT$ is invariant under $\mathbb{Z}[\sigma]$-equivariant chain homotopy.
\end{lem}

\begin{proof}
%

Suppose that $f: C \rightarrow E$ and $g: E \rightarrow C$ are inverse up to chain homotopy of $\mathbb{Z}[\sigma]$-modules.  Then $f^{\sigma - 1}: C^{\sigma - 1} \rightarrow E^{\sigma - 1}$ and $g^{\sigma-1}:E^{\sigma - 1} \rightarrow C^{\sigma - 1}$ are chain homotopy inverses.  Likewise, $f^{P(\sigma)}: C^{P(\sigma)} \rightarrow E^{P(\sigma)}$ and $g^{P(\sigma)}: E^{P(\sigma)} \rightarrow C^{P(\sigma)}$ are chain homotopy inverses.  Therefore, 
$$f^{\sigma - 1}: H^{*}(C^{\sigma - 1}) \xrightarrow{\sim} H^{*}(E^{\sigma - 1})$$  
$$f^{P(\sigma)}: H^{*}(C^{P(\sigma)}) \xrightarrow{\sim} H^{*}(E^{P(\sigma)})$$ 
are isomorphisms.  By the five lemma,  
$$f: H^{*}(C / (C^{\sigma-1} \oplus C^{P(\sigma)})) \xrightarrow{f} H^{*}(E/ (E^{\sigma - 1} \oplus E^{P(\sigma)}))$$
is an isomorphism too.  The result follows.

%
\end{proof}

\begin{lem}[(b) additivity]  Let $C_0$ and $C_1$ be bounded chain complexes of finite free abelian groups.  Suppose that either $\chi(C_0) = 0$ or $\chi(C_1) = 0.$  Then
$$\log NRT_{\sigma}((C_0 \oplus C_1)^{\otimes p}) = \log NRT_{\sigma}(C_0^{\otimes p}) + \log NRT_{\sigma}(C_1^{\otimes p}).$$
\end{lem}

\begin{proof}
We will actually prove this equality for each curly brace occurring in the definition of $NRT_{\sigma}$ individually.

\begin{itemize}
\item
The equality 
\begin{eqnarray*}
(p-1)[(C_0 \oplus C_1)^{\otimes p}]^{\sigma - 1} - [(C_0 \oplus C_1)^{\otimes p}]^{P(\sigma)} &=& (p-1) [C_0^{\otimes p}]^{\sigma - 1} - [C_0^{\otimes p}]^{P(\sigma)} \\
&+&  (p-1) [C_1^{\otimes p}]^{\sigma - 1} - [C_1^{\otimes p}]^{P(\sigma)}.
\end{eqnarray*}
in the Grothendieck group of $\mathbb{Z}$-modules can be checked after making a finite flat base change to $R = \mathbb{Z}[\mu_p].$  We can expand
$$(C_0 \oplus C_1)^{\otimes p} = \oplus_{\epsilon} C_{\epsilon_1} \otimes ... \otimes C_{\epsilon_p} =: \oplus_{\epsilon}C_{\epsilon},$$
where $\epsilon = (\epsilon_1,...,\epsilon_p)$ runs over all binary sequences of length $p.$  $\sigma$ acts by cyclic permutation on $(C_0 \oplus C_1)^{\otimes p}$ and so permutes the above summands in an evident manner.  Consider a $\sigma$-orbit $\mathcal{O}$ of summands 
$$C_{\mathcal{O}} = C_{\epsilon} \oplus ... \oplus C_{\sigma^{p-1} \epsilon}$$
where not all $\epsilon_i$ are equal.  Clearly, for every $p$th root of unity $\zeta,$ the group
$$C_{\mathcal{O}}^{\sigma - \zeta} \cong C_{\epsilon}$$
by projection.  Thus,
\begin{eqnarray*}
(p-1)C_{\mathcal{O}}^{\sigma - 1} - C_{\mathcal{O}}^{P(\sigma)} &=& (p-1) C_{\mathcal{O}}^{\sigma - 1} - \oplus_{\zeta \in \mu_p} C_{\mathcal{O}}^{\sigma - \zeta} \\
&=& (p-1) C_{\epsilon} - \oplus_{\zeta \in \mu_p} C_{\epsilon} \\
&=& 0
\end{eqnarray*}
in the Grothendieck group of $R$-modules.  Additivity for the first curly-braced term follows.     

\item
For any complex $E$ which is $\mathbb{Z}[\sigma]$-free, 
$$\sum {}^{*} \log |E / (E^{\sigma - 1} \oplus E^{P(\sigma)})| = \log p \sum {}^{*} \text{rank}_{\mathbb{Z}[\sigma]}(E_i).$$
But the complex
$$C_{\mathcal{O}} = C_{\epsilon} \oplus ... \oplus C_{\sigma^{p-1} \epsilon},$$
where not all $\epsilon_i$ are equal, is free over $\mathbb{Z}[\sigma]$ of rank $\text{rank}_{\mathbb{Z}}C_{\epsilon}.$  Thus,
\begin{eqnarray*}
\sum {}^{*} \log |C_{\mathcal{O}} / (C_{\mathcal{O}}^{\sigma - 1} \oplus C_{\mathcal{O}}^{P(\sigma)})| &=& \log p \sum {}^{*} \text{rank}_{\mathbb{Z}} C_{\epsilon}\\
&=& \log p \; \chi(C_{\epsilon}) \\
&=& \log p \; \prod \chi(C_{\epsilon_i}) \\
&=& 0.
\end{eqnarray*}
\end{itemize}
Putting these two calculations together, we see that indeed
$$\log NRT_{\sigma}( (C_0 \oplus C_1)^{\otimes p}) = \log NRT_{\sigma}(C_0^{\otimes p}) + \log NRT_{\sigma}(C_1^{\otimes p}).$$
\end{proof}

\section{Corrections and Improvements} \label{corrections}

Our proof of Corollary \ref{twistedatequalstwistedrt} proceeded in two steps: 
\begin{itemize}
\item[(a)]
Prove that 
\begin{equation} \label{comparisontoproducterror}
\tau_{\sigma}(\M,\loc) - RT_{\sigma}(\M,\loc) = \tau_{\sigma}(\M_{\sigma}^p, L^{\boxtimes p}) - RT_{\sigma}(\M_{\sigma}^p,L^{\boxtimes p}).
\end{equation}
This is Theorem \ref{finalmariage}. Our proof of \eqref{comparisontoproducterror} was cavalier about the connected components of $\M_{\sigma}.$  Nonetheless, \eqref{comparisontoproducterror} can still be salvaged.  
\item[(b)]
Prove under the hypotheses of Corollary \ref{twistedatequalstwistedrt} that
\begin{equation}
 \tau_{\sigma}(\M_{\sigma}^p, L^{\boxtimes p}) - RT_{\sigma}(\M_{\sigma}^p,L^{\boxtimes p}) = 0
\end{equation} 
for a special choice of Morse theroetic data implicit in the definition of $RT_{\sigma}.$  While this happens to be true, we'll explain how the proof described in \S \ref{product} relies on some unchecked compatibilies in Morse theory.
\end{itemize}

In \S \ref{stepa}, we discuss how to salvage Step (a) by a variation on the argument from Proposition \ref{isometric}.  In \S \ref{stepb}, we discuss the subtleties surrounding Step (b).  In \S \ref{errortermzero}, we discuss a direct approach to understanding the error term in the Bismut-Zhang formula.  For readers only interested in the statement of Corollary \ref{twistedatequalstwistedrt}, we \emph{strongly recommend proceeding directly to \S \ref{errortermzero}}.
 
\subsection{Step (a), Proposition \ref{isometric}, and Remark \ref{ssnormalbundle}} \label{stepa}

\subsubsection{Proposition \ref{isometric} and Remark \ref{ssnormalbundle}}
We let the isometry group of a symmetric space act on the right for consistency with  \cite{Rohlfs}.  \\
Let $\mathbf{H}$ be real semisimple group, $H = \mathbf{H}(\RR), K \subset H$ a maximal compact subgroup and $X = K \backslash H$ its associated symmetric space.  Suppose $\sigma$ is an automorphism of $\mathbf{H}$ of finite order $p$ normalizing $K$ and acting isometrically on $X.$  Let $\Gamma \subset H$ be a cocompact, $\sigma$-stable, torsion-free lattice.  Let $\M = K \backslash H / \Gamma.$

Let $\M_0$ be a connected component of the fixed point set $\M_\sigma.$  Fix a base point $p \in \M_0,$ a lift $\widetilde{p} \in X$ and let $i_0: \widetilde{\M_0} \rightarrow X$ be the resulting totally geodesic inclusion of the universal cover of $\M_0$ into $X.$  The heart of Proposition \ref{isometric} is showing that 
\begin{equation} \label{ssinclusion} 
i_0: \widetilde{\M_0} \rightarrow X \text{ is isometric to the inclusion } \Delta: \widetilde{\M_0}  \xrightarrow{\text{diagonal}} \widetilde{\M_0}^p.
\end{equation}
In general, \eqref{ssinclusion} is false.   For example, no simple symmetric space $X$ admits a product structure.  In particular, the statement from Remark \ref{ssnormalbundle} ``The above identification of normal bundles carries through in full generality" is misleading.  It is also disconcerting that $\M_{\sigma}$ is not necessarily equidimensional, making the meaning of some things written in Proposition \ref{isometric} unclear.  

However, in the first application intended for this paper \cite{Lip2}, we took $\mathbf{H} = R_{E/F} \G,$ where $F$ is imaginary quadratic, $E/F$ is a cyclic Galois extension of prime degree, and $\G = \PGL_1(D)$ for a quaternion algebra $D/F.$  We use the formalism of \cite{Rohlfs} to explain why Proposition \ref{isometric} is okay in this situation.  In particular, we'll explain why \eqref{ssinclusion} holds.

Let $\Gamma \subset H$ be $\sigma$-stable.  Then $\sigma$ acts isometrically on $\Gamma \backslash X = \M.$  For simplicity, suppose $\sigma$ is an involution.  Let 
\begin{align*}
Z^1(\sigma,\Gamma) &= \{ \delta \in \Gamma:  \delta \sigma(\delta) = 1\} \\
H^1(\sigma, \Gamma) &= Z^1(\sigma,\Gamma) / ( \delta \sim \gamma^{-1} \delta \sigma(\gamma) ).
\end{align*}

Every $\delta \in Z^1(\sigma, \Gamma)$  gives rise to a different automorphism of $\mathbf{H}$:  
$$\sigma_\delta(g) := \delta \sigma(g) \delta^{-1}.$$ 

The self-isometry of $X$
\begin{eqnarray*}
\sigma_\delta : X &\rightarrow& X \\
x &\mapsto& \sigma(x) \cdot \delta^{-1}
\end{eqnarray*}
acts compatibly: 
$$\sigma_\delta(x \cdot g) = \sigma_\delta(x) \cdot \sigma_{\delta}(g) \text{ for all } g \in H, x \in X.$$
In particular, $\sigma_{\delta}$ induces a self-isometry of $X / \Gamma$ and $\sigma_{\delta}(x\Gamma) = \sigma(x\Gamma)$ for all $x \in X.$  

Let $X_{\delta}$ and $\Gamma_{\delta}$ respectively denote the fixed point sets of $\sigma_{\delta}$ acting on $X$ and $\Gamma.$  Let $F_{\delta}$ denote the image of the natural map
$$X_{\delta} / \Gamma_{\delta} \xrightarrow{\cong} F_{\delta} \subset (X / \Gamma)^{\sigma}$$
As Rohlfs explains \cite[Proposition 1.3]{Rohlfs},

$$(X / \Gamma)^{\sigma} = \bigsqcup_{\delta \in H^1(\sigma,\Gamma)} F_{\delta}.$$

In particular, if $H^1(\sigma, H) = 0,$ then every $X_{\delta}$ is connected and $F_{\delta}$ are the connected components of $(X / \Gamma)^{\sigma}.$

Suppose that $H^1(\sigma,H) = 0.$  Then every $\delta \in Z^1(\sigma,H)$ can be expressed as $\delta = \gamma^{-1} \sigma(\gamma).$  The element $\gamma$ yields an isometry
\begin{eqnarray*}
X_{\delta} &\rightarrow& X_1 = X^{\sigma} \\
x &\mapsto& x \cdot \gamma^{-1}
\end{eqnarray*}
induced by the translation by $\gamma^{-1}$ self-isometry of $X.$  It follows that the normal bundles of the inclusions $X_{\delta} \subset X$ and $X_1 = X^{\sigma} \subset X$ are isometric. 

We finally note that if $H = G^p$ and $\sigma$ acts by cyclic permutation, then $H^1(\sigma,H) = 0,$ the symmetric spaces $X_H = X_G^p$ are identified, and Proposition \ref{isometric} follows.

\subsection{Step (b)} \label{stepb}
This step is problematic.  After the comparison from Theorem \ref{finalmariage}, we must compute 
$$\tau_{\sigma}(\M_{\sigma}^p, L^{\boxtimes p}) - RT_{\sigma}(\M_{\sigma}^p,L^{\boxtimes p}),$$
having no control on the metric implicit in defining the analytic torsion term $\tau_{\sigma}$ or over the Morse theoretic data implicit in defining $RT_{\sigma}.$  As explained in Theorem \ref{finalmariage}, one can appeal to the anomaly formula \cite[Theorem 0.1]{BZ2} to prove metric independence of $RT_{\sigma},$ at least when all components of the fixed point set are odd dimensional.  Having odd dimensional fixed point sets is the situation of most interest for applications to twisted torsion growth. 

In our attempted proof of Theorem \ref{productcheegermuller}, we tried to compute $RT_{\sigma}(M_{\sigma}^p,L^{\boxtimes p})$ by relating it to the twisted torsion of a product cellulation, which is straightforward to relate to the Reidemeister torsion of $(M,L).$  By taking the two-fold barycentric subdivision $(K^p)''$ of a product cellulation $K^p$ of $M^p$ \cite[III,Proposition 1.1]{Bredon}, which is equivariantly chain-homotopic to $K^p,$ one obtains a $\langle \sigma \rangle$-CW triangulation in the sense of Illman \cite{Ill}.  In the proof of Theorem \ref{productcheegermuller}, we then appealed to the main result of Illman \cite{Ill} to relate Reidemeister torsion computed from $(K^p)''$ and that computed from the abstract Morse data we cannot control: for every finite group $U,$ two smooth triangulation of a $U$-manifold admit a smooth common refinement.  

This proof would succeed if we knew that the Reidemeister torsion of the chain complex given by the abstract Morse data were equivariantly chain isomorphic (or at least equivariantly chain homotopic) to the $G$-cellulation of $M^p$ given by the unstable manifolds of the Morse data.  When $G = 1,$ this compatibility is known for a special choice of Morse data \cite[\S 2.5]{Nic}.  We have not verified the required compatibility for $G \neq 1$ and the Morse data implicit in Theorem \ref{finalmariage}.  

\subsection{Direct understanding of the Bismut-Zhang error term} \label{errortermzero}
Fortunately, we can circumvent all difficulties from \S \ref{dRecmt} and \S \ref{product} by understanding the error term in the Bismut-Zhang formula directly.

\begin{prop} \label{betterBZ}
Let $(\M,g)$ be an arbitrary compact, Riemannian manifold.  Let $\loc \rightarrow M$ be a metrized, unimodular local system.  Let $\sigma$, of finite order, act equivariantly on $\loc \rightarrow M$ by isometries.  Suppose every component of the fixed point set $\M_\sigma$ of $\M$ is odd dimensional.  Then there is a choice of Morse data $(f,X)$ on $\M$ satisfying the hypotheses of Theorem \ref{BZ} for which
$$\tau_{\sigma}(\M,\loc,g) = RT_{\sigma}(\M,\loc; f,X).$$
\end{prop}

\begin{rem} \label{metricsforbetterBZ}
No acyclicity assumption appears in Proposition \ref{betterBZ}. The quantity $RT_{\sigma}(\M,\loc;f,X)$ depends on an auxillary choice of metrics on $H^\bullet(\M,\loc)^{\sigma-1}$ and $H^\bullet(\M,\loc)^{P(\sigma)}$ for all $i.$  It is for the metric induced by $\loc$-valued harmonic forms with their $L^2$-metric that Proposition \ref{betterBZ} holds. 
\end{rem}

\begin{proof}
Reprise the notation of Theorem \ref{BZ}.  As explained in Lemma \ref{unimodular}, the ``first half" of the error term in the Bismut-Zhang formula vanishes because the differential form $\theta_{\sigma}(\loc,h_\loc)$ is identically zero.  The Bismut-Zhang formula thus simplifies to
\begin{align} \label{morsepartoferrorterm}
&\tau_{\sigma}(\M,\loc,g) - RT_{\sigma}(\M,\loc; f,X) \nonumber \\
&= - \frac{1}{4} \sum_{x \in \emph{Crit}(f) \cap \M_{\sigma}} (-1)^{\ind(f|_{\M_{\sigma}},x)} \sum_j (n_{+}(\beta_j,x) - n_{-}(\beta_j,x)) \cdot C_j \cdot \tr [\sigma | \loc_x]. 
\end{align}
We apply this formula for the special Morse data constructed in \cite[Theorem 1.8]{BZ2} and \cite[Theorem 1.10]{BZ2}.  Specifically, 
\begin{itemize}
\item[($\alpha$)]
Construct an invariant Morse function $f$ whose Hessian is \emph{positive definite on the normal bundle to} $\M_\sigma.$

\item[($\beta$)]
Modify the metric $g$ to a metric $g'$ which equals $g$ in a neighborhood of all critical points of $f$ in such a way that $X = \nabla_{g'}(f)$ satisfies Morse-Smale transversality.
\end{itemize}

For this particular choice of Morse data, the expression
$$\sum_j (n_{+}(\beta_j,x) - n_{-}(\beta_j,x)) \cdot C_j \cdot \tr [\sigma | \loc_x]$$
is constant for critical points $x$ in a single connected component $\M_0$ of the fixed point set $\M_\sigma.$  Indeed, $n_{+}(\beta_j,x) = \dim N(\beta_j)|_{\M_0}, n_{-}(\beta_j,x) = 0,$ and $\tr(\sigma | L_x)$ is a continuous integer valued function of $x \in \M_0$ and hence is constant.  Therefore,
\begin{align*}
& - \frac{1}{4} \sum_{x \in \emph{Crit}(f) \cap \M_{\sigma}} (-1)^{\ind(f|_{\M_{\sigma}},x)} \sum_j (n_{+}(\beta_j,x) - n_{-}(\beta_j,x)) \cdot C_j \cdot \tr [\sigma | \loc_x] \\
&= \sum_{ \M_0 \in \pi_0(\M_\sigma)} \text{constant}(\M_0) \sum_{x \in Crit(f) \cap \M_0}  (-1)^{\ind(f|_{\M_{\sigma}},x)} \\
&= \sum_{\M_0 \in \pi_0(\M_\sigma)} \text{constant}(\M_0) \chi(\M_0) \\
&= 0.
\end{align*}
Thus, 
$$\tau_{\sigma}(\M,\loc,g') - RT_{\sigma}(\M,\loc; f,X) = 0.$$
To conclude, note that by the anomaly formula of Bismut-Zhang \cite[Theorem 0.1]{BZ2},
$$\tau_{\sigma}(\M,\loc,g') = \tau_{\sigma}(\M,\loc,g)$$
when all components of the fixed point set are odd-dimensional.
\end{proof}

\begin{rem}
Assume $\M$ is connected.  If $\sigma$ is an involution and $\M_\sigma \neq \emptyset,$ the orientation of $\sigma$ equals $(-1)^{\mathrm{codim}(\M_0)}$ for every connected component of $\M_0.$  In particular, all components $\M_0$ have dimension of the same parity.  This dimension is odd if $\M$ is odd-dimensional and $\sigma$ is orientation-preserving or if $\M$ is even-dimensional and $\sigma$ is orientation-reversing.  
\end{rem}

\subsection{Corollary \ref{twistedrtlss} holds for local systems which are not rationally acyclic}
Let $A^\bullet = MS^\bullet(X,\loc_{\RR})$ be the Morse-Smale complex for the metrized local system $\loc_{\RR} \rightarrow \M$ of $\RR$-vector spaces acted on equivariantly by an isometry $\sigma$ of order $p.$  Defining the equivariant Reidemeister torsion  $RT_{\sigma}(A^\bullet)$ requires volume forms on every $A^i$ and volume forms on every $H^i(A^\bullet).$  The volume forms on every $A^i$ were described in the discussion preceeding Definition \ref{morsesmalert}.  If $\loc_{\RR}  \rightarrow \M$ is acyclic, Definition \ref{morsesmalert} is complete.  This acyclicity assumption was enforced whenever convenient throughout the paper, since it was required for all of our intended applications at the time.  We define the Reidemester torsion of $A^\bullet$ more completely now.  

\begin{defn} \label{morsesmalertwithregulator}
Notation as above.  We define $RT_{\sigma}(X, \loc_{\RR})$ to be the $\sigma$-equivariant Reidemeister torsion of the $\sigma$-complex $A^\bullet := \MS^{\bullet}(X,\loc_{\RR})$ with the metric described in the discussion preceeding Definition \ref{morsesmalert}  \textcolor{blue}{together with metrics on $H^*(A^\bullet[\sigma - 1])$ and $H^*(A^\bullet[P(\sigma)])$ induced by the $L^2$-metric on $\loc_{\RR}$-valued harmonic forms on $\M$.}
\end{defn}

As Remark \ref{metricsforbetterBZ} explains, it is for this choice of metrics that Proposition \ref{betterBZ} holds.

For the rest of this section, suppose $\loc_\RR$ is the base change to $\RR$ of an equivariant local system of free ableian groups $\loc \rightarrow \M.$

\begin{defn} \label{twistedregulator}
Let $R^i(\M,\loc)^{\sigma - 1}$ and $R^i(\M,\loc)^{P(\sigma)}$ respectively denote the volume of the lattices $H^i(\M,\loc)^{\sigma - 1} \subset H^i(\M,\loc_{\RR})^{\sigma- 1}$ and $H^i(\M,\loc)^{P(\sigma)} \subset H^i(\M,\loc_{\RR})^{P(\sigma)}$ where the vector spaces are equipped with the inner product induced by the $L^2$-metric on $\loc_{\RR}$-valued harmonic forms.  We define
$$R^i_{\sigma}(\M,\loc) := \frac{R^i(\M,\loc)^{\sigma-1}}{ \left( R^i(\M,\loc)^{P(\sigma)} \right)^{\frac{1}{p-1}}}$$ 
to be the \emph{twisted regulator} of $(\M,\loc)$ in degree $i.$
\end{defn}

\begin{prop} \label{twistedrtlsswithregulator}
Let $\M$ be a locally symmetric space.  Let $\loc \rightarrow \M$ be a \textcolor{blue}{not necessarily rationally acyclic} metrized, unimodular local system of free abelian groups acted on isometrically by $\langle \sigma \rangle \cong \mathbb{Z} / p\mathbb{Z}.$  Suppose that the fixed point set $\M_{\sigma}$ has Euler characteristic 0.  Then
\begin{eqnarray*}
\log RT_{\sigma}(\M,\loc) &=&  - \sum_i {}^{*} \left(\log \left|H^i(\M, \loc)_{\mathrm{tors}}[p^{-1}]^{\sigma - 1} \right|  -  \frac{1}{p-1} \log \left|H^i(\M, \loc)_{\mathrm{tors}}[p^{-1}]^{P(\sigma)} \right| \right) \\
&+& \textcolor{blue}{\sum {}^{*}  \log R^i_{\sigma}(\M,\loc) } \\ 
&+& O\left( \log|H^{*}(\mathcal{M}, \loc)_{\mathrm{tors}}[p^{\infty}]| +\log |H^{*}(\mathcal{M}, \loc_{\mathbb{F}_p})| + \log |H^{*}(M, \loc_{\mathbb{F}_p})| \right).
\end{eqnarray*}
\end{prop}

\begin{proof}
This is Corollary \ref{twistedrtlss} with a correction term for twisted regulators.  Corollary \ref{twistedrtlss} is proven in two steps:  

\begin{itemize}
\item
Corollary \ref{rtvsnrt} compares $N_\sigma = NRT_{\sigma}(\M,\loc),$ the naive twisted Reidmeister torsion of $(\M,\loc)$ to $RT_{\sigma}(\M,\loc).$  The definition of $RT_{\sigma}(\M,\loc)$ is to non-acyclic local systems from Definition \ref{morsesmalert} using harmonic metrics as described in Definition \ref{morsesmalertwithregulator}.  If we replace $N_{\sigma}$ by $N_{\sigma}'$ defined by 
$$\log N_\sigma' := \log N_\sigma + \sum {}^* \log R^i_\sigma(\M,\loc),$$ 
the proof of Corollary \ref{rtvsnrt} continues to work without change.  

\item
The analogue for $p$ of Proposition \ref{finitecomplexestimate} (see $\eqref{thirdest}_p$) compares $\log N_{\sigma}$ with the first line of the equation displayed in Proposition \ref{twistedrtlsswithregulator} with error bounded by the third line.  None of the estimates of Proposition \ref{finitecomplexestimate} are affected by the acyclicity assumption.
\end{itemize}

Incorporating the above minor change (accounting for twisted regluators), the proof for Corollary \ref{twistedrtlss} extends to non-acyclic local systems too.
\end{proof}

\nocite{astl2,perry1,pollardsag1}



\end{document}